\documentclass[]{my-interact}

\usepackage{epstopdf}
\usepackage{subcaption}
\usepackage{url}

\theoremstyle{plain}
\newtheorem{Thm}{Theorem}[section]

\theoremstyle{definition}

\theoremstyle{remark}
\newtheorem{Rem}{Remark}

\begin{document}

\articletype{RESEARCH ARTICLE}

\title{Optimal control of a heroin epidemic mathematical model}

\author{\name{P. T. Sowndarrajan\textsuperscript{a}, 
L. Shangerganesh\textsuperscript{a}, 
A. Debbouche\textsuperscript{b,c} 
and	D. F. M. Torres\textsuperscript{c}\thanks{CONTACT: D. F. M. Torres. Email: delfim@ua.pt}}
\affil{\textsuperscript{a}Department of Applied Sciences, 
National Institute of Technology, Goa 403401, India; 
\textsuperscript{b}Department of Mathematics, Guelma University, Guelma 24000, Algeria;\\
\textsuperscript{c}Center for Research and Development in Mathematics and Applications (CIDMA),
Department of Mathematics, University of Aveiro, Aveiro 3810-193, Portugal}}

\maketitle


\begin{abstract}
A heroin epidemic mathematical model with prevention information 
and treatment, as control interventions, is analyzed, assuming that
an individual's  behavioral response depends on the spreading 
of information about the effects of heroin. Such information 
creates awareness, which helps individuals to participate in preventive 
education and self-protective schemes with additional efforts. 
We prove that the basic reproduction number is the threshold 
of local stability of a drug-free and endemic equilibrium. 
Then, we formulate an optimal control problem to minimize the 
total number of drug users and the cost associated with 
prevention education measures and treatment. We prove existence 
of an optimal control and derive its characterization through
Pontryagin's maximum principle. The resulting optimality system 
is solved numerically. We observe that among all possible strategies, 
the most effective and cost-less is to implement both control policies. 
\end{abstract}


\begin{amscode}
34D20; 49J15; 92D30
\end{amscode}


\section{Introduction}

An opioid drug, made from morphine, is known as heroin. 
It is a natural substance in the seedpod of the Asian opium poppy 
plant, used as a recreational drug for its euphoric effects. 
Heroin is also known as diamorphine, used as a painkiller 
or in an opioid replacement treatment. It is a white 
or brown powder, or a black sticky substance known 
as black tar heroin \cite{NIDA}. The users inject these types of contents into a vein. 
Heroin can also be smoked, snorted or inhaled. When injected into a vein, 
the drug has two to three times the effect of a similar dose of morphine. 
Nowadays, the use of non-medical prescription drugs is becoming a significant 
threat around the world. It is estimated that about 23\% of individuals who 
take these substances become dependent on it. Further, drug users are suffering 
severe mental illness, including suicidal, especially among youth individuals. 
According to the World Heath Organization (WHO), roughly 450,000 people 
die as a result of drug use \cite{WDR2018}. 
The spread of heroin usage and its addiction follows 
many of the familiar aspects of an epidemics. 
The National Survey on Drug Use and Health (NSDUH) \cite{Lipari2015} 
estimates that the percentage of heroin users aged twelve or above was particularly
high between 2002 and 2008. Among the young adults, it is similar high from 2009 to 2012. 
Currently, many countries are affected by the heroin drug trade and its growing 
number of users. Due to high in quality and low in cost, heroin is likely 
to approach consumer markets over the world with increased consumption 
and related harmful effects. Drug users who share needles have a higher 
risk of the proliferation of other diseases, like human immunodeficiency virus (HIV), 
Hepatitis B and C \cite{Li,Garten}.   

A successful anti-drug vaccine produces an immune response to block 
the target drug from entering the brain and thus avoiding psychoactive 
or addictive effects. The experimental new vaccine for heroin addiction, 
developed at the National Institute on Drug Abuse (NIDA) \cite{NIDA}, 
incites the generation of antibodies, which block the effects of heroin.  
It prevents the drug from crossing the blood-brain barrier 
and the euphoric effects of heroin, so far tested in mice and rats. 
Side effects in quitting heroin usage are very severe and, often, 
constrain heroin addicts to reverse. Available medicines can be given 
during the detoxification stage to prevent and reduce physical symptoms.

Treatment of drug users is an expensive procedure. It is a long-term process 
involving various interventions and consistent monitoring to recover successfully.  
Many countries still fail to provide treatment and other health services, 
mainly due to lack of fund. It is also a significant burden on the health 
system of many nations. Several types of researches suggest that drug usage 
and the associated harm are highest among young people aged 18 to 25 years. 
Therefore, the rise in awareness among users is needed to make countries healthy. 
Prevention education, treatment interventions, and alternative development programs,  
as well as a criminal justice response,  prevent an increase in drug habituation 
and its disorders. Furthermore, it is also important to provide treatment 
and services to minimize the adverse health effects due to drug use. 

Mathematical models are a handy tool to describe and predict how classes 
of drug users behave. Moreover, treatment strategies can 
naturally be added to such models as control variables. 
The resulting control systems can then play a vital role 
in the understanding of the drug addiction problem.  
This is our motivation in the present work: to investigate the
real-life problem of heroin mathematically, as an epidemic problem,
with the aim to reduce the transmission mechanism. As already mentioned,
for a better control of the transmission mechanism, it is crucial 
to consider both prevention education and treatment. Moreover, to treat heroin 
addiction and disorders related to its usage, it is important 
the integration of both behavioral and pharmacological medications. 
Such integrated approach will eventually result 
in a society free of drug-addiction. With the reduction 
of the criminal behavior, one expects an increase in the employment 
rate and a lower risk of other related health diseases. Therefore, 
prevention and treatment of heroin-addiction is beneficial 
and efficient for individuals and society. 

During last few years, several mathematical models have been developed 
to describe the heroin epidemic \cite{Mulone,White}. The literature
includes clinical and theoretical studies
\cite{White,Samanta2011,Mushayabasa2015,Mushayabasa2011}, 
as well as educational campaigns \cite{UNODC2014}. In \cite{Mulone,White}, 
susceptible, untreated, and treated individuals are considered 
with a standard incidence rate. In \cite{Isaac2017}, Wangari and Stone 
model the heroin problem based on similar assumptions as used in infectious diseases
and consider a saturation function for treating the heroin drug users. 
In \cite{Wang2011},  Wang et al. consider a mass action incidence rate 
and prove the existence of a drug-free equilibrium and a unique endemic equilibrium, 
which is stable under some conditions. 
Huang and Liu study the stability for a heroin epidemic model 
with distributed delay \cite{Huang2013},
improving related results previously obtained in \cite{Liu2011}.  
A non-autonomous heroin epidemic model is also analyzed 
in \cite{Samanta2011}, where it is suggested that the spread of the heroin among
users can be controlled by full screening measures. 
More recently, Wang et al. \cite{MR3877418} analyzed mathematically an age-structured 
heroin epidemic model, which can be used to describe the spread 
of heroin habituation and addiction in a heterogeneous environment.
Even though the above-proposed models study many important features, 
most of them fail to consider the heroin epidemic model as a control system. 
Here we propose a heroin epidemic disease model with control 
variables and study it using optimal control theory.

In \cite{Mushayabasa2015}, the authors formulate a mathematical model 
for illicit drug use and investigate optimal control strategies for the model. 
Precisely, they incorporate two control functions: one to reduce the intensity 
of social influence and the other to increase the rate of detection 
and rehabilitation of illicit drug users. In \cite{Saha}, a synthetic drug 
transmission model is proposed and an optimal control problem is formulated. 
They show that a proper optimal policy can minimize the cost burden as well 
as the number of addicted individuals. In \cite{Kassa2015}, the authors 
incorporate control interventions as education campaigns and treatment 
to minimize the impact of HIV. In \cite{Saha2019}, an epidemic model 
with the effect of information about the vaccine and treatment, 
as control interventions, is considered. Optimal control of a SIR model 
with education or information, that causes a change in the behavior 
response, is proposed in \cite{Joshi2015}. The dynamics
of illicit drug use and its optimal control analysis is investigated 
in \cite{Mushayabasa2015}.  In \cite{Nicholas2017}, a mathematical model 
with prescription drug addiction and treatment is proposed to control the opioid epidemic. 
An optimal control problem and a cost-effectiveness analysis for the transmission 
of the Zika virus is analyzed in \cite{Abdulfatai2018}, with four types 
of preventive measures as control variables. In \cite{Ebenezer2016,MyID:364}, Ebola models 
with a preventive control in the form of education campaigns are investigated. 
Very recently, an age-structured heroin epidemic model is formulated 
with partial differential equations, under the assumption 
that susceptibility and recovery are age-dependent, 
keeping in view some control measures for heroin addiction and using
optimal control for simulations, which show the effect on the entire 
population \cite{MR4178232,MR4261505}.

Motivated by the above mentioned works, we extend here available heroin epidemic models 
by introducing two new compartments. In particular, we extend the model proposed in 
\cite{White} by modeling behavioral change, through the spreading of information 
and preventive education, and treatments. The novelty of the model consists to express 
the risk factors of drugs through information, as a behavioral response to susceptible 
individuals. As a consequence, in our model the susceptible have an active role 
in preventive education and provide a self-protective effect. Furthermore, 
failure in participating in preventive programs move individuals back 
to the susceptible population. 

The paper is organized as follows. In Section~\ref{s2}, we formulate a new
mathematical model for the heroin epidemic with a behavioral response. 
Furthermore, the drug free equilibrium is computed and the basic reproduction number
$R_0$ obtained. Section~\ref{s3} deals with the stability analysis of the equilibrium points
(we prove that the heroin free equilibrium is locally asymptotically stable when
the basic reproduction number is less than one, cf. Theorem~\ref{thm01},
and, under a suitable condition, is globally asymptotically stable, cf. Theorem~\ref{thm001};
and we obtain conditions under which the endemic equilibrium of the system 
is locally asymptotically stable, cf. Theorem~\ref{thm02}) and we perform
a sensitivity analysis of the epidemiological model. Then, in Section~\ref{s4}, 
we formulate an optimal control problem and do its analysis: we prove existence 
of an an optimal control pair that minimizes the proposed 
cost functional, in finite time, cf. Theorem~\ref{thm03}; 
and we characterize it using Pontryagin's Maximum Principle, 
cf. Theorem~\ref{thm:PMP}. Numerical simulations are given in Section~\ref{s5}. 
We end with some concluding remarks in Section~\ref{sec:conc}. We claim that 
the model here proposed and analyzed can also be useful
to combat other drug epidemics such as cocaine.


\section{Mathematical model}
\label{s2}

In this section, we propose and analyze a mathematical model 
for the heroin epidemics. A schematic diagram for the proposed model 
is shown in Figure~\ref{schfig}.
\begin{figure}[h!]
\centering
\includegraphics[width=7cm,height=5cm]{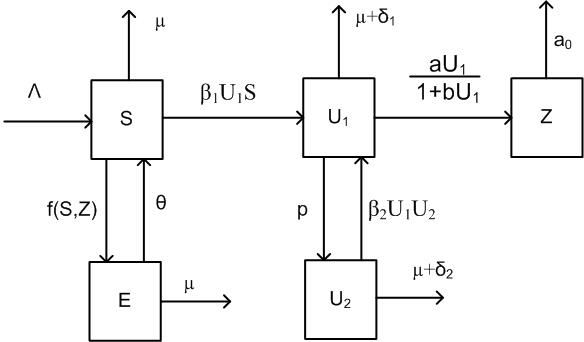}
\caption{Schematic diagram of the proposed heroin epidemic model.} 
\label{schfig}
\end{figure}
The parameter $\Lambda$ represents the constant recruitment rate 
of susceptible individuals, $S$. The natural death rate of all individuals 
is assumed to be $\mu$. The parameter $\beta_1$ represents the probability 
of becoming a drug user, $U_1$, and $\beta_2$ denotes the rate of drug users 
in treatment, $U_2$, relapsing to untreated. The rate at which drug users enter 
into treatment is $p$. Later, we will incorporate a ``case holding'' control mechanism 
in the model by replacing the parameter $p$ with $u_2(t)$, which will act 
as a control variable. The individuals in preventive education, $E$,
are in a state of self-protection, and stop to participate in preventive
education, moving back to the susceptible class, at a rate $\theta$. 
Parameters $\delta_1$ and $\delta_2$ denote induced death rates caused 
by heroin, respectively for drug users without and with treatment.
Information regarding heroin mainly spreads through various media such as TV, newspapers etc., 
as well as active social and educational campaigns from the government. This information density 
is proportional to the population density of drug users under no treatment and will change 
as population $U_1$ changes. Let $Z(t)$ denote the density of information spreading 
in the population at time $t$, such that the information density $Z(t)$ vanishes when $U_1(t)=0$.
This information increases the awareness in the behavioral change 
of insusceptible individuals to protect themselves from consuming heroin.
Even though the people are informed, everyone does not respond to it equally.
So, only a fraction of susceptible population with information 
is responding to the harmful effects of consuming heroin and changing their behavior, 
moving to class $E$. The rate of behavioral response via information interaction 
will be a function of both the densities of susceptible individuals and information, 
that is, $f(S,Z)$. Also, the growth of information is a function of $U_1$, that is,
$g(U_1)$, as the growth of information depends only on the density 
of untreated heroin consumers. The rate of behavioral response of susceptible individuals, 
caused by the information about the harmful and risk factors of heroin, 
is here assumed as $f(S, Z)=u_1\rho S Z$  \cite{Saha2019}, 
where $u_1\rho$ is the corresponding response rate and the parameter $\rho$ 
is the rate of information interaction by which individuals change their behavior. 
Later, we will consider $u_1$, the response intensity through information, 
as a control variable $u_1(t)$. The growth rate of information $g(U_1)$ is assumed as a
saturated functional of the form $g(U_1)=\frac{aU_1}{1+bU_1}$, 
where $a$ is the growth rate of information and $b$ the saturation constant. 
We denote by $a_0$ the natural degradation rate of information, 
which happens with time due to natural fading of memory about information 
as well as complacent behavior. Furthermore, we consider the following 
assumptions in our model: drug users who are not in treatment
may contact with susceptible and users in treatment in an undesirable way, 
transmitting to them the habit of consuming heroin; drug users who are in treatment 
does not transmit the habit of consuming heroin to susceptible; finally, those who
are in treatment may return back to no-treatment due to contact with $U_1$ individuals.
Thus, the proposed system of nonlinear ordinary differential equations 
is as follows:
\begin{equation}
\label{model}
\begin{cases}
\displaystyle \frac{dU_1(t)}{dt}=\beta_1S(t)U_1(t)-pU_1(t)+\beta_2U_1(t)U_2(t)
-(\mu+\delta_1)U_1(t),\\[2mm]
\displaystyle\frac{dU_2(t)}{dt}=pU_1(t)-\beta_2U_1(t)U_2(t)
-(\mu+\delta_2)U_2(t), \\[2mm]
\displaystyle\frac{dS(t)}{dt}
=\Lambda-\beta_1S(t)U_1(t)-\mu S(t)+\theta E(t)-f\left(S(t),Z(t)\right),\\[2mm]
\displaystyle\frac{dE(t)}{dt}=f\left(S(t),Z(t)\right)-(\mu+\theta) E(t), \\[2mm]
\displaystyle\frac{dZ(t)}{dt}=g\left(U_1(t)\right)-a_0 Z(t),
\end{cases}
\end{equation}
with non-negative initial conditions  $U_1(0)$, $U_2(0)$, $S(0)$, $E(0)$ and $Z(0)$.
The total population, denoted by $N(t)$, is partitioned into the four sub-classes 
of susceptible $S(t)$, drug users not in treatment $U_1(t)$,  
drug users in treatment $U_2(t)$, and educated $E(t)$:
$N(t) :=  U_1(t) + U_2(t) + S(t) +E(t)$.

Let us consider the biological feasible region by
$$
D=\left\{ (S,U_1,U_2,E,Z)\in \mathbb{R}_+^5
: 0\leq N(t)\leq \frac{\Lambda}{\mu},\ 
0<Z(t)\leq \frac{a\Lambda}{a_0 \mu} \right\}.
$$
Next, we establish the positive invariance of the feasible region $D$. 
It is trivial to show that $(S,U_1,U_2,E,Z)>0$ is positive for all time. 
Thus, 
$$
\frac{dN}{dt}=\Lambda-\mu N-\delta_1U_1-\delta_2U_2\leq \Lambda-\mu N.
$$ 
Then, a standard comparison theorem (see, e.g., \cite{Lakshmikantham}) 
can be used to show that 
$$
N(t)\leq \left(N(0)-\frac{\Lambda}{\mu}\right)
e^{-\mu t}+\frac{\Lambda}{\mu}\leq \frac{\Lambda}{\mu}
$$ 
for all $t\geq 0$.  Moreover, we have 
$$
\frac{dZ}{dt}\leq aU_1-a_0Z\leq \frac{a\Lambda}{\mu}-a_0Z.
$$ 
It follows that 
$$
Z(t)\leq \left(Z(0)-\frac{a\Lambda}{a_0\mu}\right)
e^{-a_0t}+\frac{a\Lambda}{a_0\mu}\leq \frac{a\Lambda}{a_0\mu}
$$ 
for all $t\geq 0$ and system (\ref{model}) 
is epidemiologically and mathematically well-posed.
 
 
\subsection{Drug free equilibrium} 

In order to study the behavior of the heroin model 
at its equilibrium, we set the right-hand side of all the equations 
of system (\ref{model}) to zero. It is easy to understand that  
$\tilde{U}_1=\tilde{U}_2=\tilde{E}=\tilde{Z}=0$ at the drug free state 
and, therefore, the drug free equilibrium (DFE) point of our
heroin drug model is given as 
\begin{equation}
\label{eq:E0}
E_0=\left(\tilde{S},\tilde{U_1},\tilde{U_2},
\tilde{E},\tilde{Z}\right)=\left(\frac{\Lambda}{\mu},0,0,0,0\right).
\end{equation}


\subsection{Basic reproduction number}

The basic reproduction number, denoted by $R_0$, is a threshold parameter 
used in epidemiology to measure the transmission potential of an infection. 
It is defined to be the expected number of secondary cases produced 
from a typical infected individual when introduced into a susceptible 
population during its entire period of infection. Here, $R_0$ represents 
the total number of people that each single drug user will initiate 
to drug use during their drug-using career. To obtain the basic reproduction 
number for model (\ref{model}), we use the next generation matrix method  
of \cite{van2002}. 

Let $x=(U_1,U_2,S,E,Z)^T$. 
Then, system (\ref{model}) can be written as 
$$ 
\frac{dx}{dt}=\mathcal{F}(x)-\mathcal{V}(x), 
$$
where $\mathcal{F}(x)=(\beta_1SU_1+\beta_2U_1U_2,0,0,0,0)^T$ 
is the rate of new addictions in the population and 
\begin{multline*}
\mathcal{V}(x)=\Biggl((\mu+\delta_1+p)U_1 , 
\beta_2U_1U_2+(\mu+\delta_2)U_2-pU_1 , 
\beta_1SU_1+\mu S+u_1dSZ-\theta E-\Lambda,\\ 
\mu E+\theta E-u_1dSZ ,a_0Z-\frac{aU_1}{1+bU_1}\Biggr)^T
\end{multline*}
is the rate of transfer of individuals. 
Thus, the corresponding linearized matrices are given as 
$D\mathcal{F}(x)(E_0)$ and $D\mathcal{V}(x)(E_0)$, respectively.
To obtain the basic reproduction number $R_0$, we need to consider 
only the infected components of $F:=D\mathcal{F}(x)(E_0)$ 
and $V:=D\mathcal{V}(x)(E_0)$ \cite{van2002}. Therefore, we have
\begin{eqnarray*}
F=
\left( 
\begin{array}{ccc} 
\displaystyle \frac{\beta_1 \Lambda}{\mu} 
& 0 & 0 \\ 0 & 0 & 0 \\ 0 & 0 & 0 \end{array} \right)
\end{eqnarray*}
and 
\begin{eqnarray*}
V=
\left( 
\begin{array}{ccc} 
Q_1 &  0 & 0 \\ 
\ -p & Q_2   & 0 \\ 
\ -a & 0 \  & a_0 
\end{array} \right),
\end{eqnarray*}
where 
\begin{equation}
\label{Q1Q2}
Q_1=\mu+\delta_1+p
\quad \text{ and }\quad 
Q_2=\mu+\delta_2.  
\end{equation}
Matrix $FV^{-1}$ is said to be the next generation matrix of system (\ref{model}). 
The basic reproduction $R_0$ is given as the spectral radius of $FV^{-1}$, 
and we obtain that
\begin{equation}
\label{brn}
R_0= \frac{\beta_1 \Lambda}{\mu(\mu+\delta_1+p)}.
\end{equation}


\section{Stability analysis}
\label{s3}

In this section, we discuss the stability of the equilibrium points
of model \eqref{model}.

\begin{Thm}
\label{thm01}
If $R_0 <1$, then the heroin free equilibrium $E_0$ is locally asymptotically stable. 
\end{Thm}

\begin{proof}
The Jacobian matrix of the system at $E_0$ is 
\begin{eqnarray*}
J(E_0)=
\left( \begin{array}{ccccc} -\mu 
&  -\frac{\beta_1\Lambda}{\mu} & 0 & \theta & -\frac{u_1\rho\Lambda}{\mu}\\ 
0 & -\frac{\beta_1\Lambda}{\mu}-p-(\mu+\delta_1)   & 0 & 0 & 0 \\  
0 & p & -(\mu+\delta_2) & 0 & 0 \\ 
0 & 0 & 0 & -(\mu+\theta) & \frac{u_1\rho\Lambda}{\mu} \\ 
0 & a & 0 & 0 & -a_0
\end{array} \right).
\end{eqnarray*}
The eigenvalues of the characteristic equation of 
$J(E_0)$ are 
$$
\lambda_1=-\mu, 
\quad \lambda_2=-(\mu+\theta), 
\quad \lambda_3=-a_0,
$$ 
together with the solutions of the equation 
\begin{equation}
\label{eq:ap:RH:c}
\lambda^2+\zeta_1\lambda+\zeta_2=0,
\end{equation}
where $\zeta_1=\displaystyle2\mu+p+\delta_1+\delta_2+\frac{\beta_1\Lambda}{\mu}$ and $\zeta_2=\displaystyle-\frac{\beta_1\Lambda}{\mu}\left[(\mu+\delta_2)\left(1-\frac{1}{R_0}\right)\right]$.
By the Routh--Hurtwiz condition, the characteristic equation \eqref{eq:ap:RH:c}
has two roots with negative real parts if, and only if, $\zeta_i >0$, $i=1,2$.
The result follows from the fact that $R_0<1$ implies $\zeta_i >0$, $i=1,2$. 
\end{proof}

\begin{Thm}
\label{thm001}
If $\displaystyle \frac{\beta_1 \Lambda}{\mu(\mu+\delta_1)}<1$, then
the heroin free equilibrium $E_0$ is globally asymptotically stable.
\end{Thm}
 
\begin{proof}
We consider the Lyapunov function $L(U_1,U_2)=U_1+U_2$. 
The time derivative of $L$ computed along the solutions 
of system \eqref{model} is given by
\begin{equation*}
\frac{dL}{dt}=\beta_1 S U_1-(\mu+\delta_1)U_1-(\mu+\delta_2)U_2 
\leq (\mu+\delta_1)\left[ \frac{\beta_1 \Lambda}{\mu(\mu+\delta_1)}
-1\right]U_1-(\mu+\delta_2)U_2.
\end{equation*}	
When $ \frac{\beta_1 \Lambda}{\mu(\mu+\delta_1)}<1$, 
we get $\frac{dL}{dt}< 0$. Furthermore, $\frac{dL}{dt}=0$ if and only 
if $U_1=U_2=0$. Note that the largest compact invariant set is 
$\left\{ (\tilde{S},\tilde{U_1},\tilde{U_2},\tilde{E},\tilde{Z})|\frac{dL}{dt}=0\right\}$. 
Therefore, the drug free equilibrium $E_0$ is globally asymptotically stable 
by LaSalle's invariance principle.
\end{proof}


\subsection{Endemic equilibria}

The endemic equilibrium point $E_1 = (S^*,U_1^*,U_2^*,E^*,Z^*)$ 
of system \eqref{model} is given by
\begin{eqnarray*}
S^*&=&\frac{\Lambda}{\mu}\frac{1}{R_0}\left[
\frac{Q_1(\beta_2U_1^*+\mu+\delta_2)
-\beta_2pU_1^*}{Q_1(\beta_2U_1^*+\mu+\delta_2)}\right],\\[1mm]
U_2^*&=& \frac{p U_1^*}{\beta_2U_1^*+\mu+\delta_2}, \\[1mm]
E^*&=&\frac{u_1\rho S^*aU_1^*}{a_0(1+bU_1^*)(\mu+\theta)}, \\[1mm]
Z^*&=&\frac{aU_1^*}{a_0(1+bU_1^*)},
\end{eqnarray*}
with $U_1^*$ the positive solution of equation
$A_1U_1^3+A_2U_1^2+A_3U_1+A_4=0$, where 
\begin{eqnarray*}
A_1&=&\displaystyle Q_1\beta_2a_0b(\mu+\theta)(\mu+\delta_1),\\
A_2&=& \displaystyle -a_0\left(\mu+\theta\right)bQ_1\beta_2\Lambda\left(1-\frac{1}{R_0}\right)
+Q_1\beta_2a_0(\mu+\theta)(\mu+\delta_1)+Q_1^2Q_2a_0b(\mu+\theta)\\
&&-\frac{\Lambda}{R_0}\beta_2pa_0b(\mu+\theta)
+\frac{\Lambda}{R_0}\beta_2u_1\rho a(\mu+\delta_1),\\
A_3&=& \displaystyle -a_0(\mu+\theta)Q_1\beta_2\Lambda\left(1-\frac{1}{R_0}\right)
-a_0b(\mu+\theta)Q_1Q_2\Lambda\left(1-\frac{1}{R_0}\right)+Q_1^2Q_2a_0(\mu+\theta)\\
&&-\frac{\Lambda}{R_0}\beta_2pa_0(\mu+\theta)+\frac{\Lambda}{R_0}Q_1Q_2u_1\rho a, \\
A_4&=&\displaystyle -a_0(\mu+\theta)Q_1Q_2\Lambda\left(1-\frac{1}{R_0}\right)
\end{eqnarray*}
and $Q_1$ and $Q_2$ are given as in \eqref{Q1Q2}.
For all possible values of the parameters, one has always $A_1 >0$. 
If $R_0>1$, then $A_4<0$. Now, using Descartes' rule of signs, we obtain:
\begin{itemize}
\item[(i)] if $A_2>0$ and $A_3>0$, then there exists exactly one positive root,
\item[(ii)] if $A_2>0$ and $A_3<0$, then there exists exactly one positive root,
\item[(iii)] if $A_2<0$ and $A_3>0$, then there exists three or one positive roots,
\item[(iv)] if $A_2<0$ and $A_3<0$, then there exists exactly one positive root.
\end{itemize}
Therefore, $U_1^*$ may have a non-trivial positive value if any one 
of the above four conditions (i)--(iv) is satisfied.

\begin{Thm}
\label{thm02}
The endemic equilibrium $E_1 = (S^*,U_1^*,U_2^*,E^*,Z^*)$  
of the system is locally asymptotically 
stable if $R_0>1$ and the following conditions are satisfied:
 \begin{equation}
\label{eq:RH:holds}
\begin{gathered}
\Theta_i>0, \quad i=2,\ldots,5,\\
\Theta_1\Theta_2\Theta_3>\Theta_3^2+\Theta_1^2\Theta_4,\\ 
(\Theta_1\Theta_4-\Theta_5)(\Theta_1\Theta_2\Theta_3
-\Theta_3^2-\Theta_1^2\Theta_4)>\Theta_5(\Theta_1\Theta_2-\Theta_3)^2+\Theta_1\Theta_5^2,
\end{gathered}
\end{equation}
where
\begin{equation}
\label{eq:theta:i}
\begin{split}
\Theta_1&=-a_{11}-a_{22}-a_{33}-a_{44}-a_{55}, \\
\Theta_2&=a_{11}(a_{44}+a_{33}+a_{22}+a_{55})
-a_{21}a_{12}-a_{14}a_{41}-a_{23}a_{32}+a_{55}a_{22}\\
&\quad +a_{33}(a_{22}+a_{55})+a_{44}(a_{33}+a_{22}+a_{55}), \\
\Theta_3&= a_{11}(a_{23}a_{32}-a_{33}a_{44}-a_{44}a_{22}
-a_{44}a_{55}-a_{33}a_{22}-a_{33}a_{55}-a_{55}a_{22})\\
&\quad -a_{21}a_{15}a_{52}-a_{33}a_{44}a_{22}-a_{55}(a_{33}a_{44}+a_{44}a_{22}+a_{33}a_{22})\\
&\quad+a_{21}(a_{12}a_{33}+a_{12}a_{44}+a_{12}a_{55})+a_{14}(a_{33}a_{41}+a_{41}a_{22}+a_{41}a_{55})\\
&\quad +a_{23}a_{32}(a_{44}+a_{55}),\\
\Theta_4&= a_{11}a_{44}(a_{33}a_{22}-a_{23}a_{32})-a_{11}a_{55}(a_{23}a_{32}
-a_{33}a_{44})+a_{11}a_{55}a_{22}(a_{44}+a_{33})\\
&\quad -a_{21}a_{12}(a_{33}a_{44}+a_{33}a_{55}+a_{44}a_{55})
-a_{14}a_{33}a_{41}(a_{22}+a_{55})-a_{21}a_{14}a_{45}a_{52}\\
&\quad +a_{21}a_{15}a_{52}(a_{33}+a_{44})+a_{23}a_{32}(a_{14}a_{41}-a_{44}a_{55})\\
&\quad +a_{55}a_{22}(a_{33}a_{44}-a_{14}a_{41}), \\
\Theta_5&= a_{11}a_{44}a_{55}(a_{23}a_{32}-a_{33}a_{22})
-a_{14}a_{41}a_{55}(a_{23}a_{32}-a_{33}a_{22})\\
&\quad -a_{21}(a_{15}a_{33}a_{44}a_{52}
-a_{12}a_{33}a_{44}a_{55}-a_{14}a_{45}a_{33}a_{52}),
\end{split}
\end{equation}
with
\begin{equation}
\label{eq:q:aij}
\begin{split}
a_{11}&=-\beta_1U_1^*-\mu-u_1\rho Z^*, 
\quad  a_{12}=-\beta_1S^*, 
\quad a_{14}=\theta, 
\quad  a_{15}=-u_1\rho S^*, \\
a_{21}&=\beta_1U_1^*,
\quad a_{22}=\beta_1S^*+\beta_2U_2^*-Q_1, 
\quad a_{23}=\beta_2U_1^*,\\ 
a_{32}&=p-\beta_2U_2^*, 
\quad a_{33}=-\beta_2U_1^*-Q_2, \\
a_{41}&=u_1\rho Z, 
\quad a_{44}=-(\mu+\theta), 
\quad a_{45}=u_1\rho S^*,\\ 
a_{52}&=\frac{a}{(1+bU_1^*)^2}, 
\quad a_{55}=-a_0,
\end{split}
\end{equation}
and $Q_1$ and $Q_2$ given by \eqref{Q1Q2}.
\end{Thm}

\begin{proof}
The Jacobian matrix of the system at $E_1$ is 
\begin{eqnarray*}
J(E_1)=
\left( \begin{array}{ccccc} a_{11} &  a_{12} & 0 & a_{14} & a_{15}\\ 
a_{21} & a_{22}   & a_{23} & 0 & 0 \\  
0 & a_{32} & a_{33} & 0 & 0 \\ 
a_{41} & 0 & 0 & a_{44} & a_{45} \\ 
0 & a_{52} & 0 & 0 & a_{55}
\end{array} \right),
\end{eqnarray*}
where the $a_{ij}$'s are given as in \eqref{eq:q:aij}.
The characteristic equation of $J(E_1)$ is given by 
$\lambda^5+\Theta_1 \lambda^4+\Theta_2 \lambda^3
+\Theta_2 \lambda^2+\Theta_3 \lambda+\Theta_4 =0$
with the $\Theta_i$'s defined by \eqref{eq:theta:i}.
One has $\Theta_1>0$ for all feasible $S^*$ and $U_2^*$. Therefore, 
the endemic equilibrium of the system is locally asymptotically 
stable if, and only if, the Routh--Hurwitz criterion is satisfied,
that is, conditions \eqref{eq:RH:holds} hold. 
\end{proof}


\subsection{Sensitivity analysis}

A sensitivity analysis of the epidemiological model is performed 
to determine the relative importance of the model parameters 
to the infection transmission. Such analysis is important to discover 
the parameters that have a high impact on $R_0$ and should be targeted 
by intervention strategies. The basic reproduction number \eqref{brn} 
of system \eqref{model} depends on the recruitment rate of susceptible, $\Lambda$, 
on the probability $\beta_1$ of becoming a drug user, on the natural death rate $\mu$, 
on the rate $p$ at which drug users enter into treatment, 
and on the induced death rate $\delta_1$ caused by heroin. 
Computing the partial derivatives of $R_0$ with respect 
to $\beta_1$ and $p$ gives 
\begin{eqnarray*}
\frac{\partial R_0}{\partial \beta_1}
&=& \frac{\Lambda}{\mu(\mu+\delta_1+p)}>0, \\
\frac{\partial R_0}{\partial p}
&=& -\frac{\beta_1 \Lambda}{\mu(\delta_1+\mu+p)^2} <0.
\end{eqnarray*}
Next, we examine the sensitivity of $R_0$ with respect 
to the parameters $\beta_1$ and $p$, by the method 
of Arriola and Hyman \cite{Arriola2005}: the normalized 
forward sensitivity index for each of those parameters. 
The normalized forward sensitivity index of a variable 
for a parameter is the ratio of the relative change 
in the variable to the relative change in the parameter. 
When the variable is a differentiable function of the parameter, 
then the sensitivity index may be alternatively defined using partial derivatives.
Note that to the most sensitive parameter $v$ it corresponds a normalized 
forward sensitivity index of one or minus one: $\Theta_{v}=\pm 1$. 
If $\Theta_{v}=+ 1$, then an increase (decrease) of $v$ by $x$ percent  
increases (decreases) $R_0$ by $x$ percent. If $\Theta_{v}=- 1$, then
an increase (decrease) of $v$ by $x$ percent decreases (increases) $R_0$ by $x$ percent 
\cite{Rosa2019,Silva2013}. In order to reduce the drug burden, we pay more 
attention to the highest sensitivity index parameters.
Therefore, we compute for parameters $\beta_1$ and $p$ as follows:
\begin{eqnarray*}
\Theta_{\beta_1}&=&
\left|\frac{\beta_1}{R_0}\frac{\partial R_0}{\partial \beta_1}\right|=1.
\end{eqnarray*}
It is noted that the basic reproduction number $R_0$ is most sensitive 
to changes in $\beta_1$, that is, the probability of becoming a drug user. 
If $\beta_1$ increases, then $R_0$ will increase. Similarly, if $\beta_1$ 
decreases, then $R_0$ will decrease. Next,
\begin{eqnarray*}
\Theta_{p}&=&
\left|\frac{p}{R_0}\frac{\partial R_0}{\partial p}\right|<1.
\end{eqnarray*} 
Here, $R_0$ is less sensitive to changes in $p$, the rate at which 
drug users enter into treatment.  Further, $\Theta_{p}$ suggests 
that an increment in $p$ will decrease $R_0$ and a decrease in $p$ will increase $R_0$. 
As $R_0$ is more sensitive to changes in $\beta_1$ than $p$, 
we choose to focus more on $\beta_1$. Furthermore,  
$\frac{\partial R_0}{\partial p}<0$ implies that improving 
the successful treatment rate is a successful remedy from drug addiction 
and its associated disorders. Finally, this sensitivity analysis 
tell us that preventing (or) controlling individual's from drug use 
is more effective than any other strategies.


\section{Optimal control model}
\label{s4}

In this section, we begin by formulating an optimal control problem 
with vaccination and treatment as control interventions. Then, 
we prove existence of an optimal control and we characterize
it through Pontryagin's Maximum Principle.


\subsection{The total cost functional}

Our main goal is to decrease the number of drug users 
and the cost of implementing the two control interventions. 
Therefore, we consider the following total cost functional 
$J$ to  minimize, as the weighted sum of three components:
\begin{equation}
\label{cost}
J[u_1(\cdot),u_2(\cdot)]
=\int_{0}^{t_f}[B_1 U_1+B_2 u_1^4(t)+B_3u_2^2(t)] dt \longrightarrow \min
\end{equation}
subject to the model control system 
\begin{equation}
\label{ss}
\left\{
\begin{array}{lllll}
\displaystyle\frac{dS}{dt}=\Lambda-\beta_1SU_1-\mu S+\theta E-u_1(t)\rho SZ, \\[3mm]
\displaystyle \frac{dU_1}{dt}=\beta_1SU_1-u_2(t)U_1+\beta_2U_1U_2-(\mu+\delta_1)U_1,\\[3mm]
\displaystyle\frac{dU_2}{dt}=u_2(t)U_1-\beta_2U_1U_2-(\mu+\delta_2)U_2,\\[3mm]
\displaystyle\frac{dE}{dt}=u_1(t)\rho SZ-(\mu+\theta) E,\\[3mm]
\displaystyle\frac{dZ}{dt}=\frac{a U_1}{1+b U_1}-a_0 Z,
\end{array}
\right.
\end{equation}
with fixed initial conditions 
\begin{equation}
\label{eq:OC:IC}
S(0) = S_0 >0, \,  U_1(0) = U_{1,0}>0, \, 
U_2(0) = U_{2,0} >0, \,  E(0) = E_0>0,\, 
Z(0) = Z_0>0. 
\end{equation}
Here, $t_f$ is the fixed terminal time. The detailed report 
of the three components in the cost functional (\ref{cost}) is as follows:
\begin{itemize}
\item[(i)] The cost induced by heroin burden itself is  
$\displaystyle \int_{0}^{t_f} B_1 U_1(t) dt$, which is 
proportional to the number of drug users $U_1$. It also includes 
drug-affected driving, creating an impact on the environment, 
nation's economy, individual's health loss, career loss like education, 
employment, and productivity, social care, etc. The coefficient $B_1$ 
represents the positive weight constant of the heroin drug user.

\item[(ii)] Preventive education to susceptible individuals. 
Prevention from drug abuse helps the population to live longer, 
happier and healthier. It also helps in a better growth of nation's economy,
making it stronger. Therefore, providing information about the risk 
factors behind the drug use, its associated disorders and mainly its prevention, 
like effective participation in preventive education, which includes 
self-protective schemes, makes a behavioral change among susceptible population.  
Additional efforts are needed to increase prevention and turn it more effective 
in controlling the habit of drug use. Mainly, preventive and protective factors 
should include impulse control, parental monitoring, academic competence, 
anti-drug use policies, and secure neighborhood attachment. Here, all 
the susceptible individuals are made aware of preventive education 
through spreading of information. In our model, the control variable $u_1(t)$ 
is the intensity response function through information to maximize 
the individual behavioral response and keeping cost low. 
The cost  $\displaystyle \int_{0}^{t_f} B_2 u_1^4(t) dt$ is 
involved in the process of information spreading, through preventive 
education and its participation. It may be through campaigns, 
mass media, social networks, etc. It also represents the cost 
of spreading information, which includes creating awareness about 
the high-risk factors of heroin abuse and its causes. Moreover, 
this gives information about the heroin user's behavior and, 
mainly, its protective measures. The cost is comparatively higher 
because of the additional efforts needed to convince individuals 
of a behavioral change.  Hence, we consider the non-linearity 
of order four $u_1^4(t)$ \cite{Saha2019,Grass}. It represents 
the high expenses and efforts to spread the information. Here, 
the coefficient $B_2$ represents the positive weight constant 
associated with the spreading of information.

\item[(iii)] Medical treatment to drug users. 
The drug addictions and its disorders can be lowered 
by a certain level by undergoing medical procedures. 
It involves hospitalization, diagnosis, medication, 
and other subsequent therapies, like contingency management 
psychology, motivational incentives, etc.
Here, we consider the treatment rate $p$ as the control variable 
$u_2(t)$, which measures the treatment intensity. 
The cost $ \displaystyle \int_{0}^{t_f} B_3 u_2^2(t) dt$ 
is involved in providing treatment. It represents the cost for medical 
treatment. To treat drug abuse population includes psychological 
and pharmacological medications. We consider a non-linearity of order two, 
$u_2^2(t)$, in the cost for treatment \cite{Kassa2015,kumar2017}. 
The coefficient $B_3$ represents the positive weight constant associated with treatment. 
\end{itemize}
Thus, the Lagrangian function $L$ for our optimal control problem is given by
\begin{equation}
\label{L}
L(S,U_1,U_2,E,Z,u_1,u_2)=B_1U_1+B_2u_1^4+B_3u_2^2. 
\end{equation} 
The control variables $u_1(t)$ and $u_2(t)$ of our optimal control problem 
involve the following admissible control set: 
$$ 
U_{ad} = \{(u_1, u_2)|u_i(t) \mbox{ is Lebesgue measurable on $[0,t_f]$}: 
0 \leq u_i(t) \leq u_{imax}, i=1,2 \},
$$
where $u_{1max}$ and $u_{2max}$ are fixed positive constants.
 
 
\subsection{Existence of optimal control}

In this subsection, we prove that there exists an optimal control 
pair $u_1^*$ and $u_2^*$ that minimizes 
the cost functional $J$ in finite time.

\begin{Thm}
\label{thm03}
There exists an optimal control pair $u_1^*$ and $u_2^*$ in $U_{ad}$ 
such that $J(u_1^*,u_2^*)=\min\{J(u_1,u_2)\}$, solution to the optimal
control problem (\ref{cost})--(\ref{eq:OC:IC}).
\end{Thm}

\begin{proof}
To prove existence of solution, we need to satisfy the following conditions:
\begin{itemize}
\item[(1)] The admissible set of controls $U_{ad}$ 
and the state solutions of (\ref{ss}) is nonempty.
\item[(2)] The control set $U_{ad}$ is closed, 
convex and the state system can be expressed as a linear function 
of the control variables with coefficients that 
depend on time and state variables.
\item[(3)] The integrand $L$ in the cost functional (\ref{cost}) 
is convex on the control set $U_{ad}$  
and $L(S,U_1,U_2,E,Z,u_1,u_2)\geq h(u_1,u_2)$, 
where $h(u_1,u_2)$ is continuous and 
$|(u_1,u_2)|^{-1}h(u_1,u_2) \rightarrow \infty$ 
whenever $|(u_1,u_2)|\rightarrow \infty$, with
$|\cdot|$ the $L^2(0,t_f)$ norm.
\end{itemize}
For each control variables $u_1$ and $u_2$ in the set $U_{ad}$, 
the solutions of the system (\ref{ss}) are bounded and the right-hand side 
satisfies the Lipschitz condition with respect to the state variables. Therefore, 
by applying the Picard--Lindel\"{o}f theorem \cite{Coddington}, 
condition (1) holds. By definition, the control set $U_{ad}$ is closed and convex. 
The model system (\ref{ss}) is linear in the control variables $u_1$ and $u_2$ with
the coefficients dependent on the state variables. Thus, condition (2) is satisfied. 
Finally, the integrand $L$ is convex due to the biquadratic nature of $u_1$ 
and the quadratic nature of $u_2$. We have $L(S,U_1,U_2,E,Z,u_1,u_2) 
\geq B_2u_1^4+B_3u_2^2$. Let $c=\min(B_2,B_3)>0$ and $h(u_1,u_2)=c(u_1^4+u_2^2)$. Then, 
condition (3) is also satisfied. Hence, from \cite{Gaff2011}, there exists a control 
pair $u_1^*$ and $u_2^*$ such that $J(u_1^*,u_2^*) = \min  J(u_1 , u_2 )$. 
\end{proof}


\subsection{Characterization of optimal control functions}

Now, we derive necessary optimality conditions 
using Pontryagin's Maximum Principle (PMP) \cite{Lenhart,Pontryagin}. In particular,
we characterize the optimal control pair $u_1^*$ and $u_2^*$ for problem 
(\ref{cost})--(\ref{eq:OC:IC}). 

\begin{Thm}
\label{thm:PMP}	
Let $u_1^*$ and $u_2^*$ be optimal controls 
of problem (\ref{cost})--(\ref{eq:OC:IC}) 
and $S^*$, $U_1^*$, $U_2^*$, $E^*$, $Z^*$ the corresponding 
optimal state trajectories satisfying (\ref{ss})--(\ref{eq:OC:IC}). 
Then, there exists an adjoint variable 
$\lambda=(\lambda_1,\lambda_2,\lambda_3,\lambda_4,\lambda_5) \in \mathbb{R}^5$ 
that satisfies the following equations:
\begin{equation}
\label{as}
\begin{cases}
\displaystyle \frac{d\lambda_1}{dt}
=\lambda_1\beta_1U_1+\lambda_1\mu+\lambda_1u_1\rho Z
-\lambda_2\beta_1U_1-\lambda_4u_1\rho Z, \\[0.3cm]
\displaystyle \frac{d\lambda_2}{dt}=\lambda_1\beta_1S-\lambda_2\beta_1S
+\lambda_2u_2-\lambda_2\beta_2U_2+\lambda_2\mu+\lambda_2\delta_1-\lambda_3u_2 \\[0.3cm]
\hspace{0.5in}+\lambda_3\beta_2U_2-\lambda_5\frac{a}{(1+bU_1)^2}-B_1, \\[0.3cm]
\displaystyle \frac{d\lambda_3}{dt}
=-\lambda_2\beta_2U_1+\lambda_3\beta_2U_1+\lambda_3\mu +\lambda_3\delta_2, \\[0.3cm]
\displaystyle \frac{d\lambda_4}{dt}=\lambda_4\mu+\lambda_4 \theta-\lambda_1 \theta, \\[0.3cm]
\displaystyle \frac{d\lambda_5}{dt}=\lambda_1u_1\rho S-\lambda_4u_1\rho S+\lambda_5a_0,
\end{cases}
\end{equation}
with transversality conditions 
\begin{equation}
\label{tc}
\lambda_i(t_f)=0, \quad i=1,\ldots,5.
\end{equation}
Moreover, the optimal controls $u_1^*$ and $u_2^*$ are given as 
\begin{equation}
\label{oc}
\begin{split}
u_1^*(t) &= \min\left\{ \max\left\{ 0, \left( \displaystyle \frac{\rho
S^*(t)Z^*(t)}{4B_2}(\lambda_1(t)-\lambda_4(t))\right)^\frac{1}{3}\right\}, 
u_{1max} \right\}, \\
u_2^*(t) &= \min\left\{ \max\left\{ 0,  
\displaystyle\frac{(\lambda_2(t)-\lambda_3(t))U_1(t)}{2B_3}\right\}, u_{2max} \right\}.
\end{split}
\end{equation}
\end{Thm}

\begin{proof}
We define the Hamiltonian function as follows:
\begin{equation*}
\begin{split}
H(S,U_1,U_2,E,Z,u_1,u_2, \lambda)
&= L(S,U_1,U_2,E,Z,u_1,u_2)\\
&\quad +\lambda_1 \left(\Lambda-\beta_1SU_1-\mu S+\theta E-u_1\rho SZ\right)\\
&\quad +\lambda_2 \left(\beta_1SU_1-u_2U_1+\beta_2U_1U_2-(\mu+\delta_1)U_1\right)\\
&\quad +\lambda_3 \left(u_2U_1-\beta_2U_1U_2-(\mu+\delta_2)U_2\right)\\
&\quad +\lambda_4 \left(u_1\rho SZ-(\mu+\theta)\right)\\
&\quad +\lambda_5 \left(\frac{a U_1}{1+b U_1}-a_0 Z\right),
\end{split}
\end{equation*} 
where $L$ is the Lagrangian function (\ref{L}). 
Let $u_1^*$ and $u_2^*$ be the optimal controls and 
$S^*,U_1^*,U_2^*,E^*,Z^*$ the corresponding optimal state variables. 
From PMP, there exist functions $\lambda_1,\ldots,\lambda_5$ 
that satisfy the adjoint equations
$$ 
\frac{d\lambda_1}{dt}=-\frac{\partial H}{\partial S}, \ \  
\frac{d\lambda_2}{dt}=-\frac{\partial H}{\partial U_1}, \ \ 
\frac{d\lambda_3}{dt}=-\frac{\partial H}{\partial U_2}, \ \ 
\frac{d\lambda_4}{dt}=-\frac{\partial H}{\partial E}, 
\ \ \frac{d\lambda_5}{dt}=-\frac{\partial H}{\partial Z},
$$
evaluated at the optimal controls and corresponding state variables,
subject to the transversality conditions $\lambda_i(t_f)=0$, $i=1,\ldots,5$.
Therefore, we obtain the adjoint system \eqref{as} 
and terminal conditions (\ref{tc}). Finally, having in mind that
$$ 
\frac{\partial H}{\partial u_1}=4B_2u_1^3-\lambda_1\rho ZS+\lambda_4 \rho SZ
$$ 
and 
$$
\frac{\partial H}{\partial u_2}=2B_3u_2-\lambda_2U_1+\lambda_3U_1, 
$$
we obtain from the minimality condition of PMP that (\ref{oc}) holds.
\end{proof}

Concluding, the optimality conditions consist of the state system (\ref{ss}) 
with given initial conditions \eqref{eq:OC:IC}, the adjoint system (\ref{as}) 
with transversality conditions (\ref{tc}), and the optimal control functions (\ref{oc}). 
In the next Section~\ref{s5}, we solve numerically the obtained optimality system:
\begin{equation}
\label{os}
\begin{cases}
\displaystyle\frac{dS(t)}{dt}
=\Lambda-\beta_1SU_1-\mu S+\theta E-u_1\rho SZ, \\[2mm]
\displaystyle \frac{dU_1(t)}{dt}
=\beta_1SU_1-u_2U_1+\beta_2U_1U_2-(\mu+\delta_1)U_1, \\[2mm]
\displaystyle\frac{dU_2(t)}{dt}
= u_2U_1-\beta_2U_1U_2-(\mu+\delta_2)U_2, \\[2mm]
\displaystyle\frac{dV(t)}{dt}
= u_1\rho SZ-(\mu+\theta) E, \\[2mm]
\displaystyle\frac{dZ(t)}{dt}
=\displaystyle\frac{a U_1}{1+b U_1}-a_0 Z, \\[2mm]
\displaystyle\frac{d\lambda_1}{dt}
=\lambda_1\beta_1U_1+\lambda_1\mu+\lambda_1u_1\rho Z
-\lambda_2\beta_1U_1-\lambda_4u_1\rho Z,  \\[2mm]
\displaystyle\frac{d\lambda_2}{dt}
=\lambda_1\beta_1S-\lambda_2\beta_1S+\lambda_2u_2
-\lambda_2\beta_2U_2+\lambda_2\mu+\lambda_2\delta_1-\lambda_3u_2 \\ 
\qquad +\lambda_3\beta_2U_2-\lambda_5\displaystyle\frac{a}{(1+bU_1)^2}-B_1, \\[2mm]
\displaystyle\frac{d\lambda_3}{dt}
=-\lambda_2\beta_2U_1 
+\lambda_3\beta_2U_1+\lambda_3\mu +\lambda_3\delta_2,  \\[2mm]
\displaystyle \frac{d\lambda_4}{dt}
=\lambda_4\mu+\lambda_4 \theta-\lambda_1 \theta,  \\[2mm]
\displaystyle \frac{d\lambda_5}{dt}
=\lambda_1u_1\rho S-\lambda_4u_1\rho S+\lambda_5a_0, \\[2mm]  
S(0) = S_0, \ U_1(0) = U_{10},\ 
U_2(0)= U_{20},\ E(0)= E_{0},\ Z(0) = Z_{0},\\[2mm]
\lambda_i(t_f)=0, \ i=1,\ldots,5,
\end{cases}
\end{equation}
with
\begin{equation}
\label{extremal:controls}
\begin{split}
u_1 &= \min\left\{ \max\left\{ 0, \left( \displaystyle \frac{\rho
S Z}{4B_2}(\lambda_1-\lambda_4)\right)^\frac{1}{3}\right\}, u_{1max} \right\},\\ 
u_2 &= \min\left\{ \max\left\{ 0,  \displaystyle\frac{(\lambda_2
-\lambda_3)U_1}{2B_3}\right\}, u_{2max} \right\}.
\end{split} 
\end{equation}

\begin{Rem}
In principle, there is a possibility of having ``singular controls'',  
which may occur along the arcs for which either $\lambda_1(t)- \lambda_4(t)$ 
or $\lambda_2(t)- \lambda_3(t)$ or both vanish. In our numerical simulations
such possibility was not found. 
\end{Rem}


\section{Numerical results and discussion}
\label{s5}

We begin by illustrating our Theorem~\ref{thm01} numerically.
For that, we consider the parameter values as given 
in Table~\ref{Table.}, for which the basic reproduction 
\eqref{brn} is less than one.
\begin{table}[h!]
\centering
\caption{Parameter values for which $R_0 <1$,
used to obtain Fig.~\ref{fig:Stable Equilibrium}.}
\label{Table.}
\begin{tabular}{llll}
\hline\noalign{\smallskip}
Parameter & Value \\
\noalign{\smallskip}\hline\noalign{\smallskip}
$\Lambda$   & 2.0    \\
$\beta_1$&  0.0002   \\
$\beta_2$ & 0.0001  \\
$\mu$   & 0.125  \\
$\rho$&   0.04  \\
$\delta_1$& 0.05  \\
$\delta_2$& 0.06  \\
$\theta$& 0.001  \\
$a$& 0.01    \\
$b$& 1.0   \\
$a_0$& 0.06  \\
\noalign{\smallskip}\hline
\end{tabular}
\end{table} 
In agreement with Theorem~\ref{thm01}, we see in
Fig.~\ref{fig:Stable Equilibrium} the stability of 
the population around the drug free equilibrium \eqref{eq:E0}.
\begin{figure}[h!]
\centering
\includegraphics[scale=0.4]{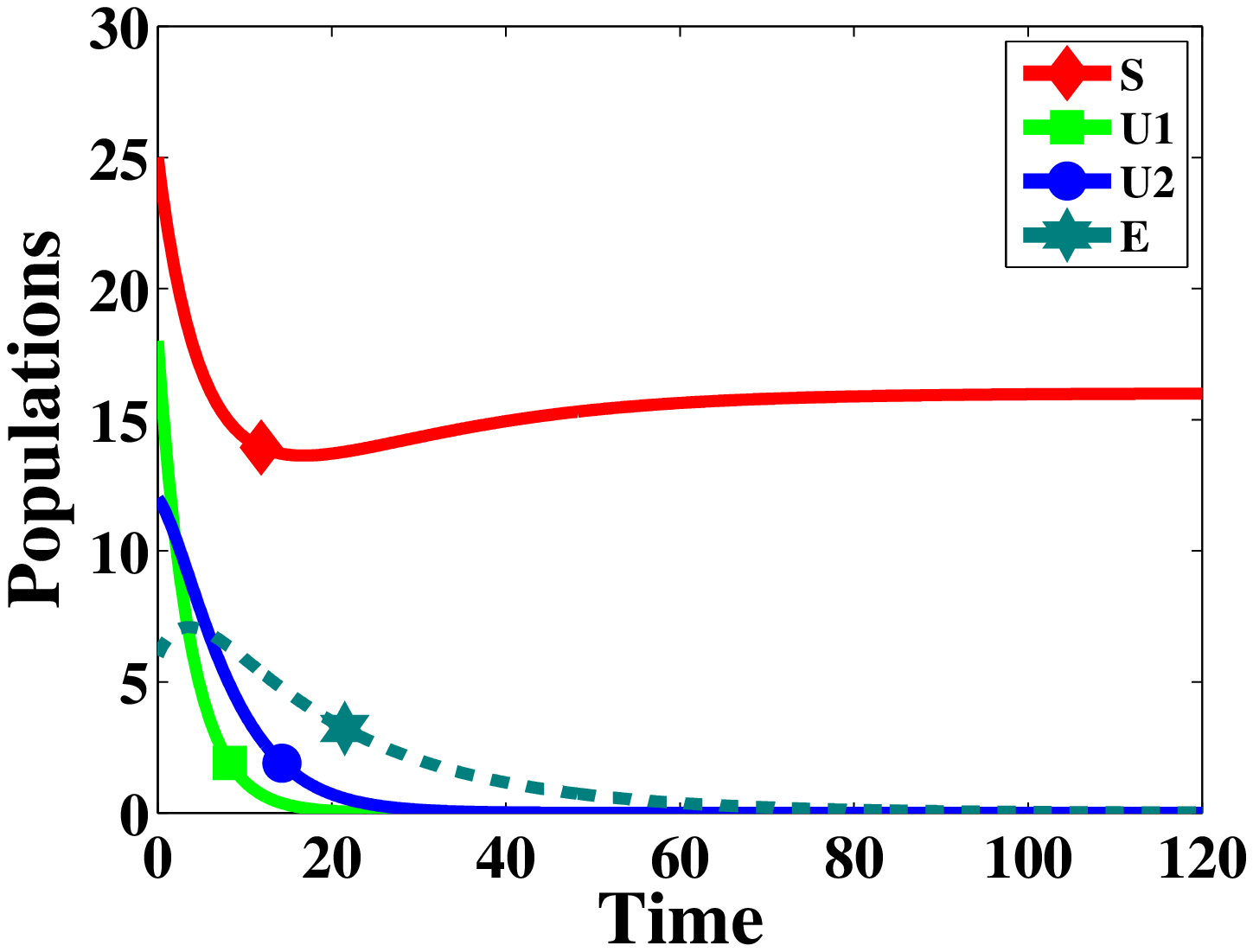}
\qquad
\includegraphics[scale=0.4]{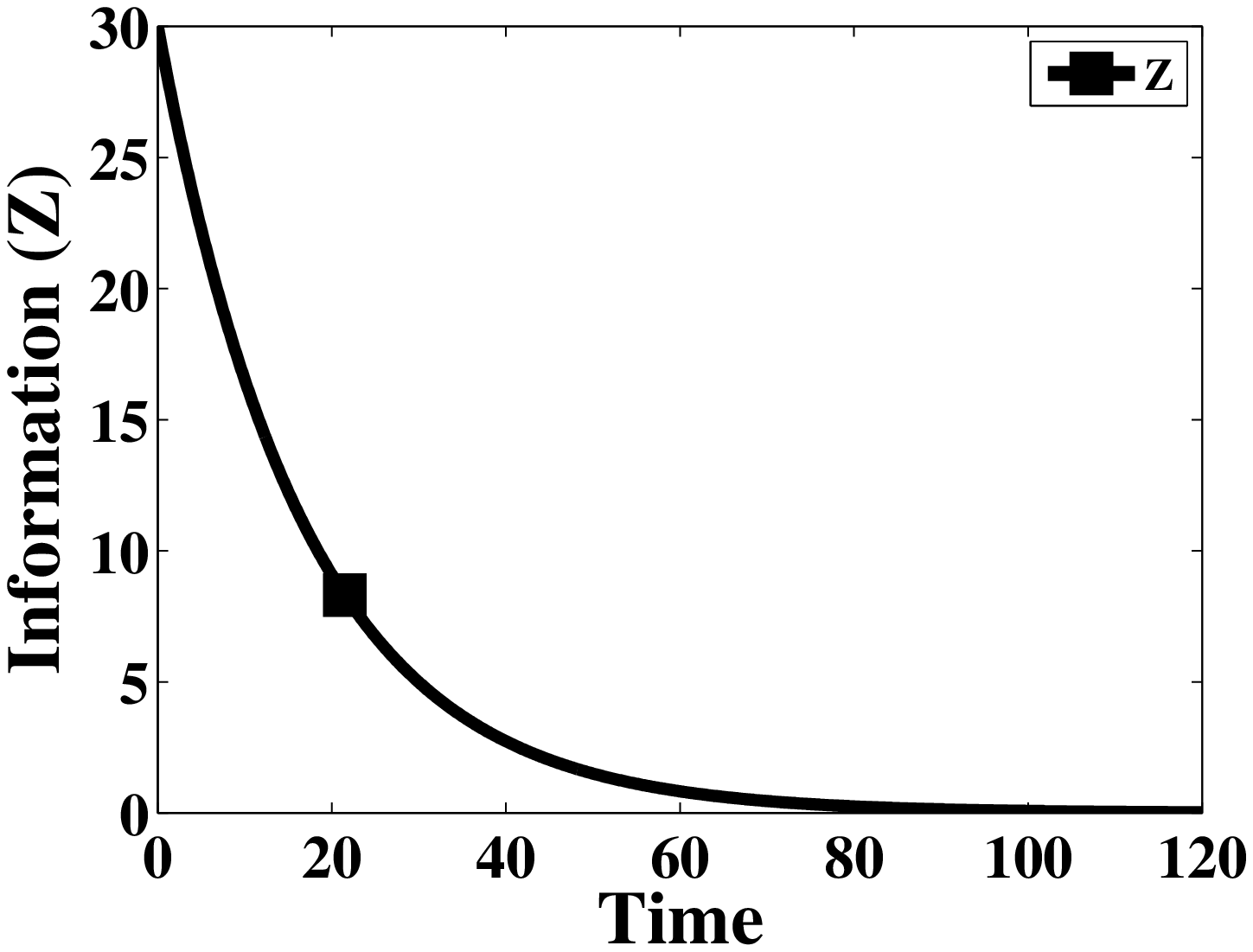}
\caption{Stability of the population around 
the drug-free equilibrium $E_0$
in agreement with Theorem~\ref{thm01}.} 
\label{fig:Stable Equilibrium}
\end{figure}

We are, however, more interested to illustrate numerically
our analytical findings and the involvement of control variables 
in the system dynamics in the endemic situation, when $R_0 > 1$,
which is more challenging and where control measures 
are crucial and optimal control theory has an important role.
We use a fourth-order Runge--Kutta algorithm to perform the numerical 
simulation of the optimal system (\ref{os}) with \eqref{extremal:controls}. 
Choosing the initial conditions 
for the states and the initial guesses for the controls, the state system (\ref{ss}) 
is solved forward in time using a fourth-order Runge--Kutta scheme.
Using the current iteration solution of the state equation (\ref{ss}) 
and the transversality conditions (\ref{tc}), the adjoint system (\ref{as}) 
is solved backwards in time by the fourth-order Runge--Kutta scheme. 
We repeat the iteration process by updating the controls using the state 
and adjoint values. This process will continue until the values of the state, 
adjoint, and controls converge.
The initial values of unknowns we have used are given by 
$S(0)=15.0$, $U_1(0)=5.0$, $U_2(0)=2.0$, $E(0)=1.25$ and $Z(0)=1.0$. 
Population profiles and control interventions are plotted 
for a time period of $t_f=30$ days.  We use the set of parameter values 
as in the Table~\ref{Table} to determine the numerical simulation 
of the optimality system with a small time step size $\Delta t = 0.03$.
\begin{table}[h!]
\centering
\caption{Values of the parameters used to illustrate 
the endemic situation and optimal control.}
\label{Table}
\begin{tabular}{llll}
\hline\noalign{\smallskip}
Parameter & Value & Units & Reference\\
\noalign{\smallskip}\hline\noalign{\smallskip}
$\Lambda$   & 0.7    & persons per day & Assumed\\
$\beta_1$&  0.01  & per day & Assumed  \\
$\beta_2$ & 0.0008 & per day & Assumed \\
$\mu$   & 0.07 & per day & Assumed \\
$\rho$&   0.04  & per day & Assumed \\
$\delta_1$& 0.05  & persons per day & \cite{Wang2011}  \\
$\delta_2$& 0.06  & persons per day & \cite{Wang2011}  \\
$\theta$& 0.001 & per day & \cite{Saha2019} \\
$a$& 0.01  & $-$ & \cite{Saha2019} \\
$b$& 1.0  & $-$ & \cite{Saha2019} \\
$a_0$& 0.06  & $-$ & \cite{Saha2019} \\
$B_1$& 6  & $-$ & Assumed  \\
$B_2$& 120  & $-$ & Assumed  \\
$B_3$& 30  & $-$ & Assumed  \\
$u_{1max}$& 1.0  & $-$ & Assumed  \\
$u_{2max}$& 1.0  & $-$ & Assumed  \\
\noalign{\smallskip}\hline
\end{tabular}
\end{table}  
Note that here the positive weights in the objective functional \eqref{cost}
are assumed. In general, the individual's response to behavioral change 
for a large population is very challenging, and it is expensive. 
Therefore, and exactly because information spreading for the community to change 
their behavior is expensive and a difficult task, we assume
the positive weight for the control $u_1$ to be higher than the control 
$u_2$ \cite{Kassa2015}. We define the following three approaches to examine 
the efficiency of the control policies introduced:
\begin{itemize}
\item {\bf Case 1} -- implementation of the optimal control variable $u_1^*$ only ($u_2 \equiv 0$);
\item {\bf Case 2} -- implementation of the optimal control variable $u_2^*$ only ($u_1 \equiv 0$);
\item {\bf Case 3} -- implementation of both optimal control variables $u_1^*$ and $u_2^*$.
\end{itemize}

Before studying these three situations under optimal control,
we illustrate the endemic situation under investigation
in Fig.~\ref{fig:without controls}, which shows the population densities 
of system (\ref{model}) without controls,
\begin{figure}[hb!]
\centering
\includegraphics[scale=0.4]{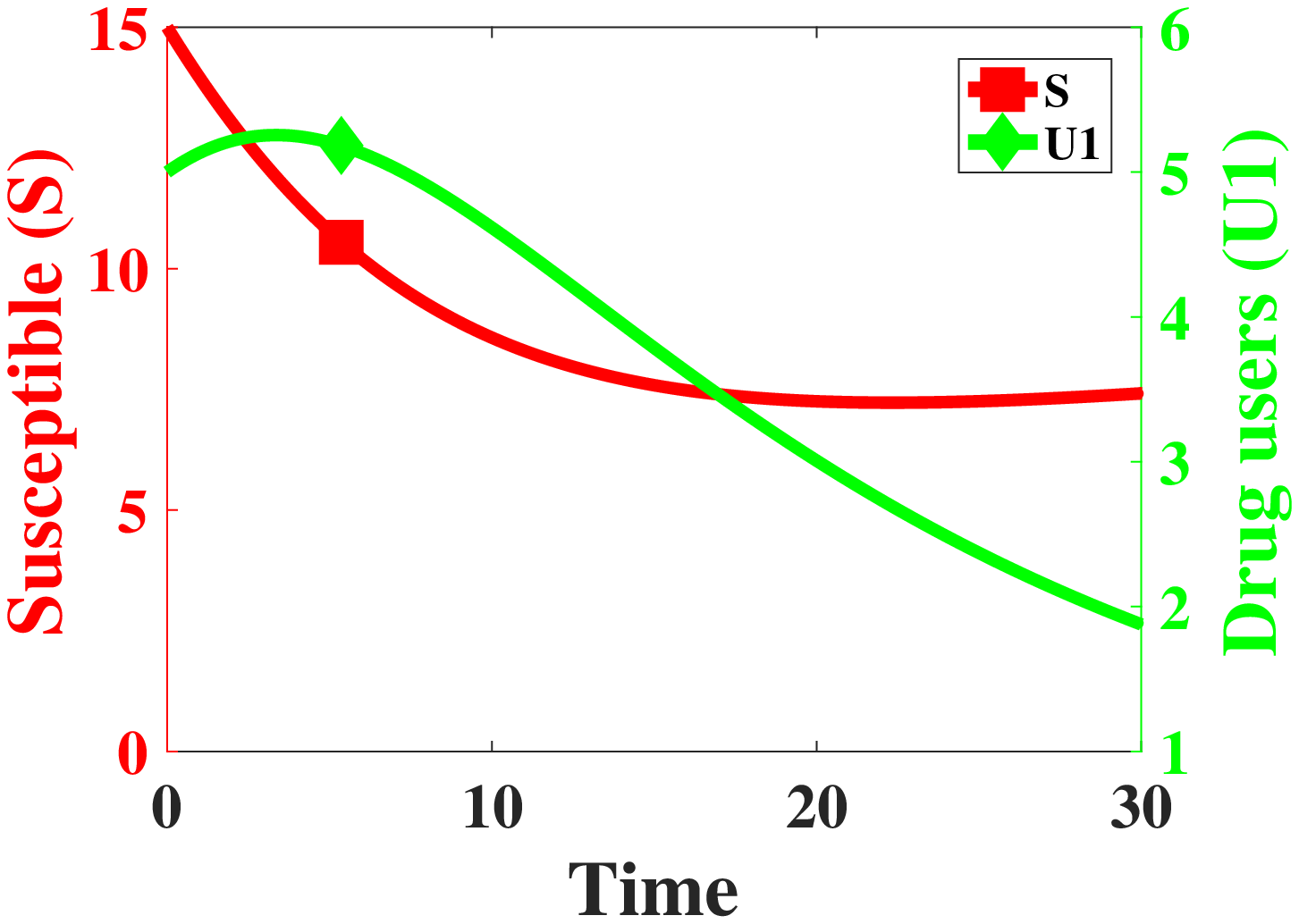}
\qquad
\includegraphics[scale=0.4]{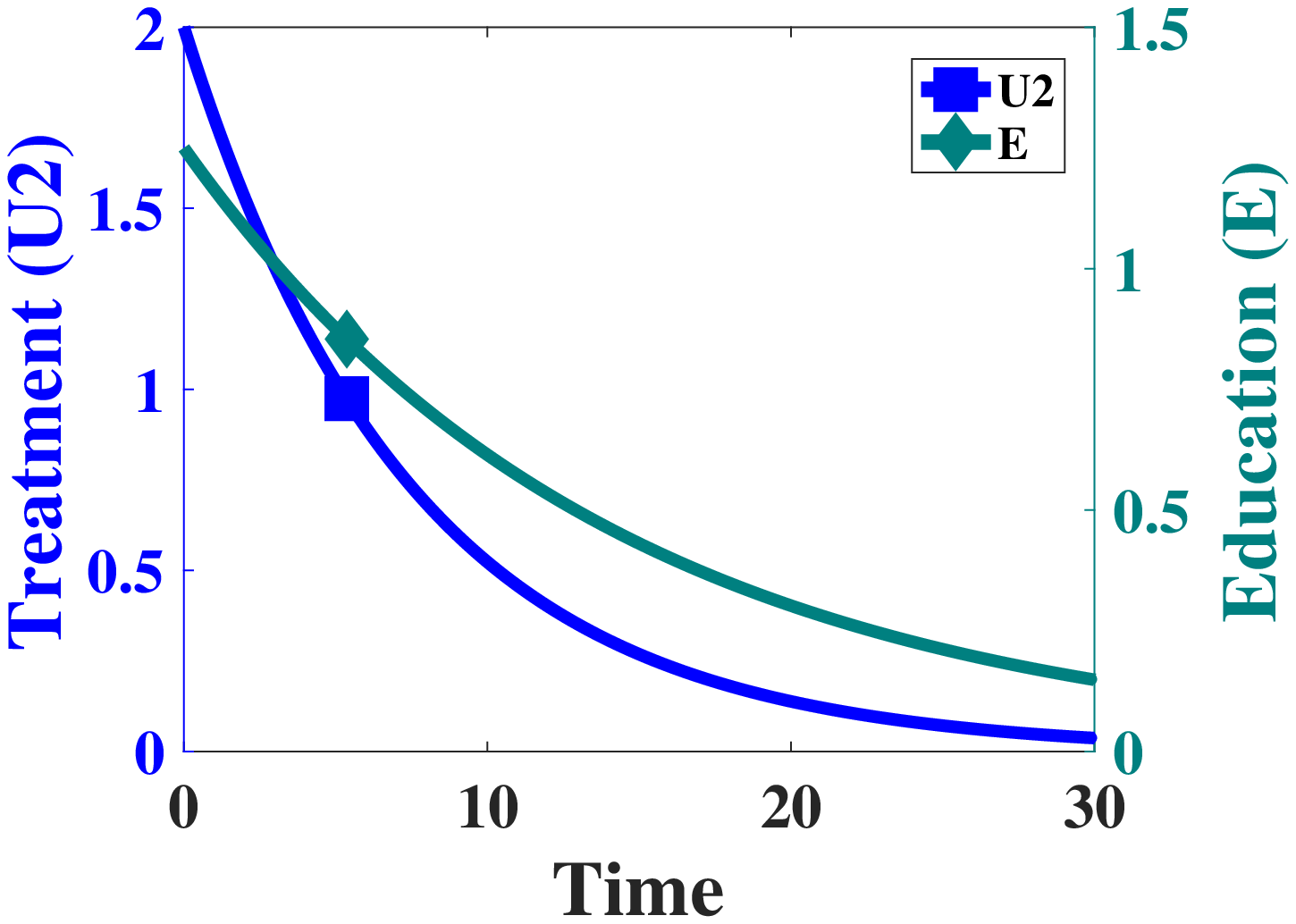}
\caption{Population profiles for the endemic situation 
of Table~\ref{Table} without controls, i.e., $u_1=u_2 \equiv 0$.} 
\label{fig:without controls}
\end{figure}
that is, where we assume $u_1 \equiv 0$ and $u_2 \equiv 0$ 
in the control system \eqref{ss}.
We observe from Fig.~\ref{fig:without controls}, and in contrast
with the situation of Fig.~\ref{fig:Stable Equilibrium}, that the number 
of drug users begins increasing and does not go to zero due to the absence 
of control variables. Therefore, the increase in heroin users creates 
an economic burden on the nation, including drug-affected driving 
and impact on the environment. It also creates opportunity loss, 
and it is the main burden to individual users with weight loss, 
mental disorders, etc. So, we aim to minimize the drug addiction 
burden and also the cost of the control policies. We induce the 
two control interventions, (i)  preventive education as vaccination 
for drug abuse, which spreads through the information and makes 
the behavioral change and (ii) treatment with medications and other therapies.

\bigskip
    
{\bf{Case 1:}} Using the same parameters as in Table~\ref{Table}, 
the above mentioned positive weights and initial conditions, we solve 
the system numerically with control $u_1$ to discuss the 
effectiveness of the control intervention $u_1$. 
The corresponding evolution in the population densities of the system is shown 
in Fig.~\ref{fig:only:u1}. 
\begin{figure}[ht!]
\centering
\begin{subfigure}[t]{0.48\textwidth}
\centering
\includegraphics[scale=0.40]{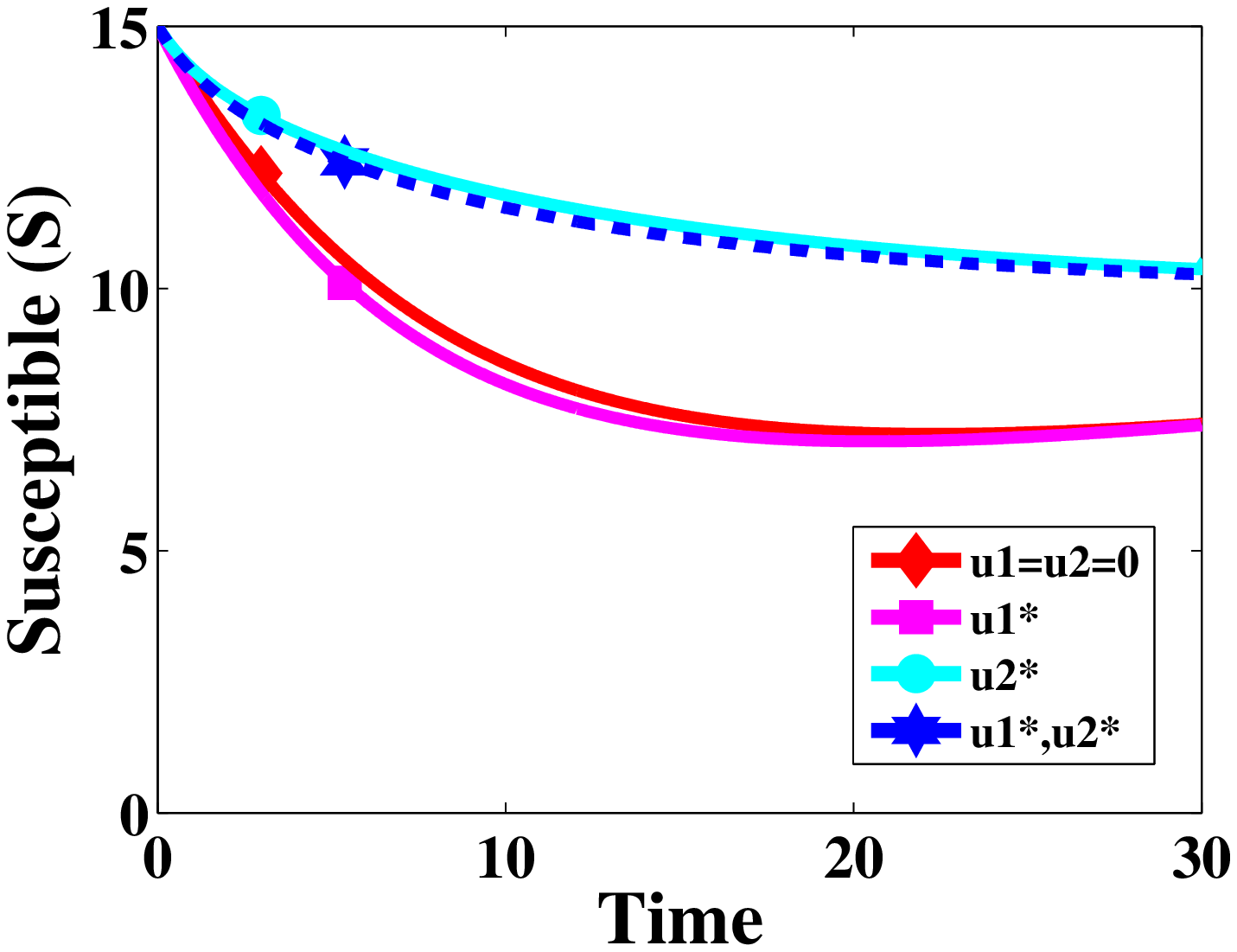}
\caption{}
\label{fig:a}
\end{subfigure}
\begin{subfigure}[t]{0.48\textwidth}
\centering
\includegraphics[scale=0.40]{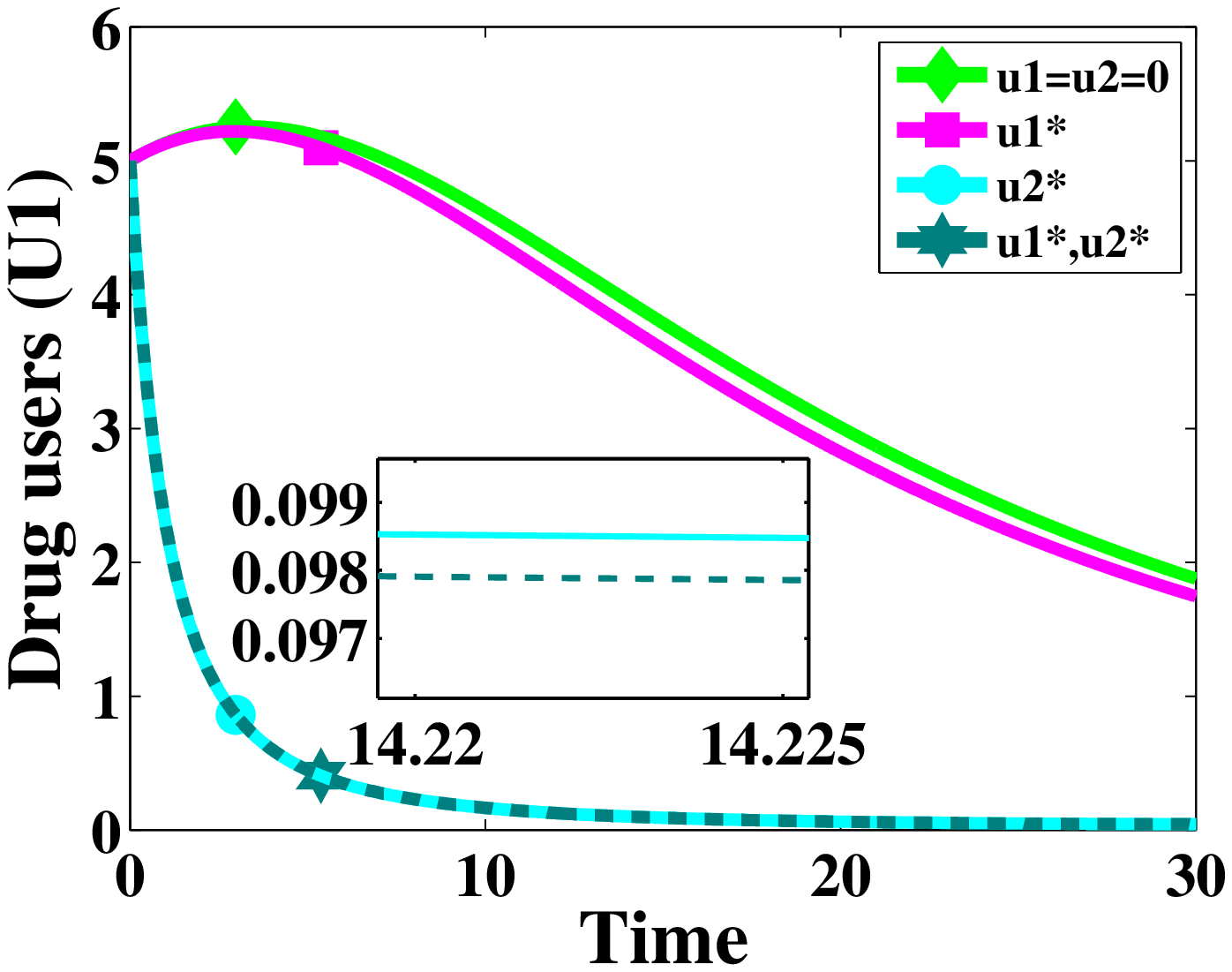}
\caption{}
\label{fig:b}
\end{subfigure}
\begin{subfigure}[t]{0.48\textwidth}
\centering
\includegraphics[scale=0.40]{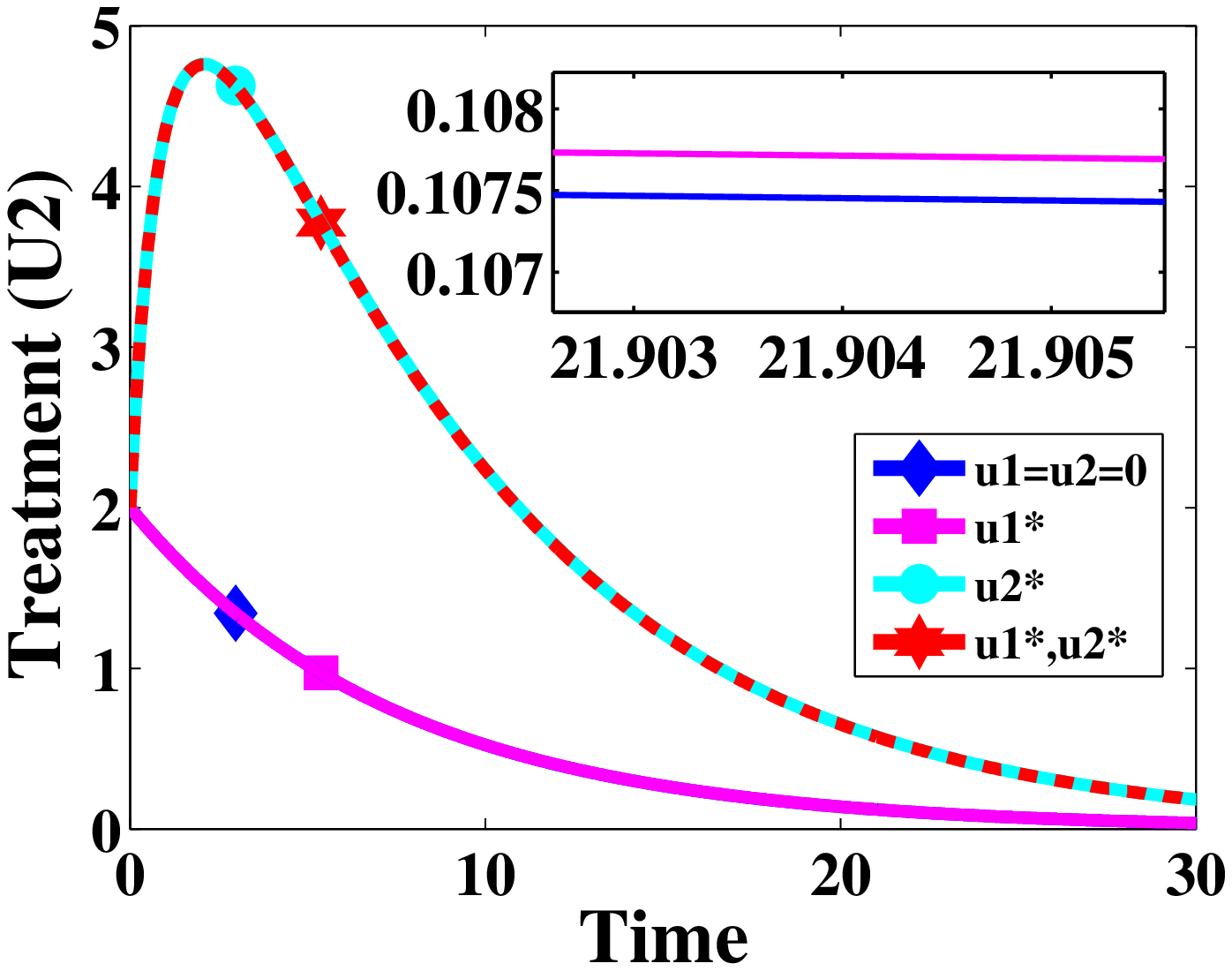}
\caption{}
\label{fig:c}
\end{subfigure}
\begin{subfigure}[t]{0.48\textwidth}
\centering
\includegraphics[scale=0.40]{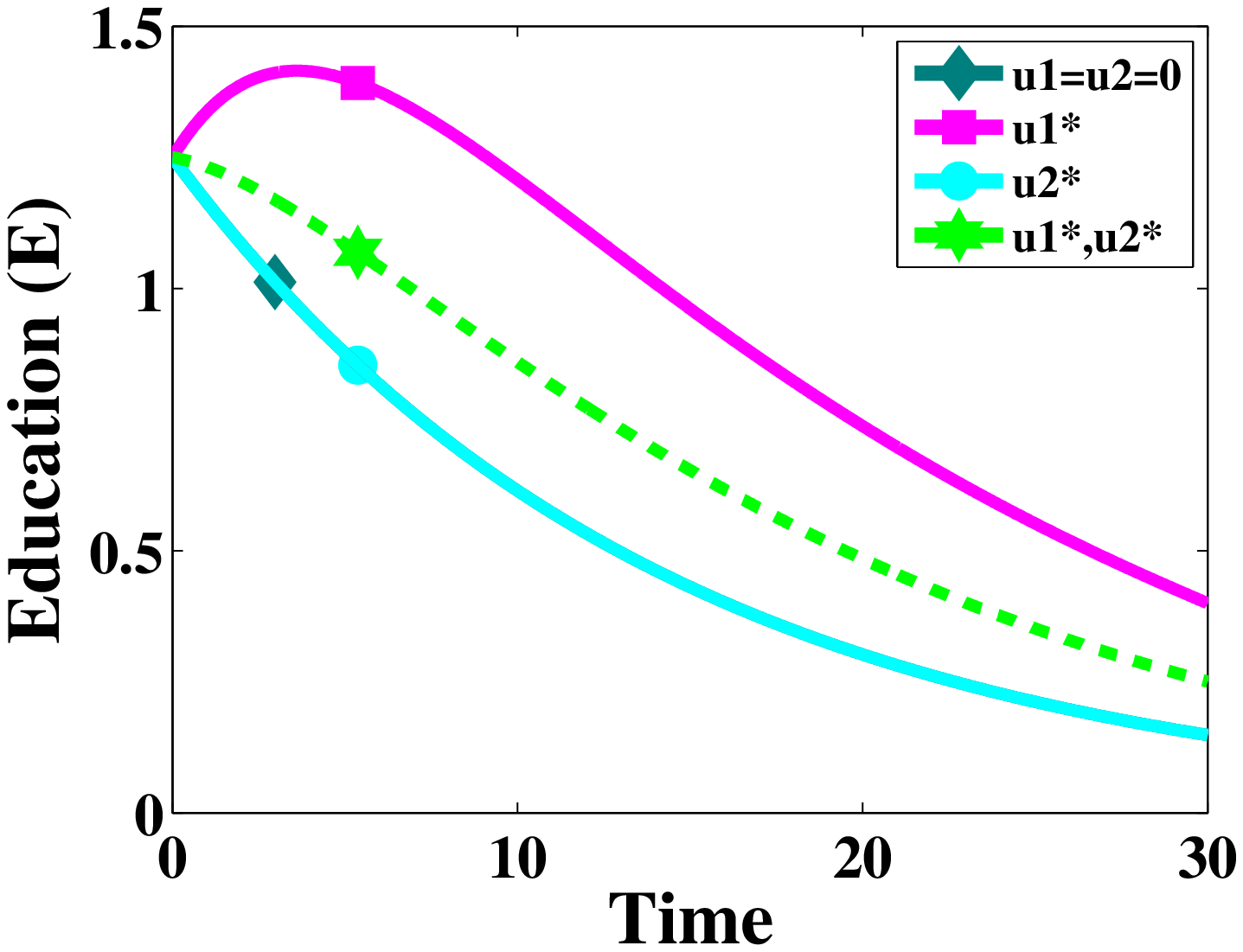}
\caption{}
\label{fig:d}
\end{subfigure}
\caption{Population profiles without controls 
and under optimal control strategies for Cases~1, 2 and 3.} 
\label{fig:only:u1}
\end{figure}
We observe from Fig.~\ref{fig:d}, that the population 
of preventive education is gradually increasing with the control 
$u_1$ than without control. Moreover, we also observe from 
Fig.~\ref{fig:a} and \ref{fig:b} a moderate decrease in the population 
of susceptible and drug users. Further, various computational results 
were carried out with different values of $\rho$, which are depicted 
in Fig.~\ref{fig:5a}. 
\begin{figure}[ht!]
\centering
\begin{subfigure}[t]{0.48\textwidth}
\centering
\includegraphics[scale=0.40]{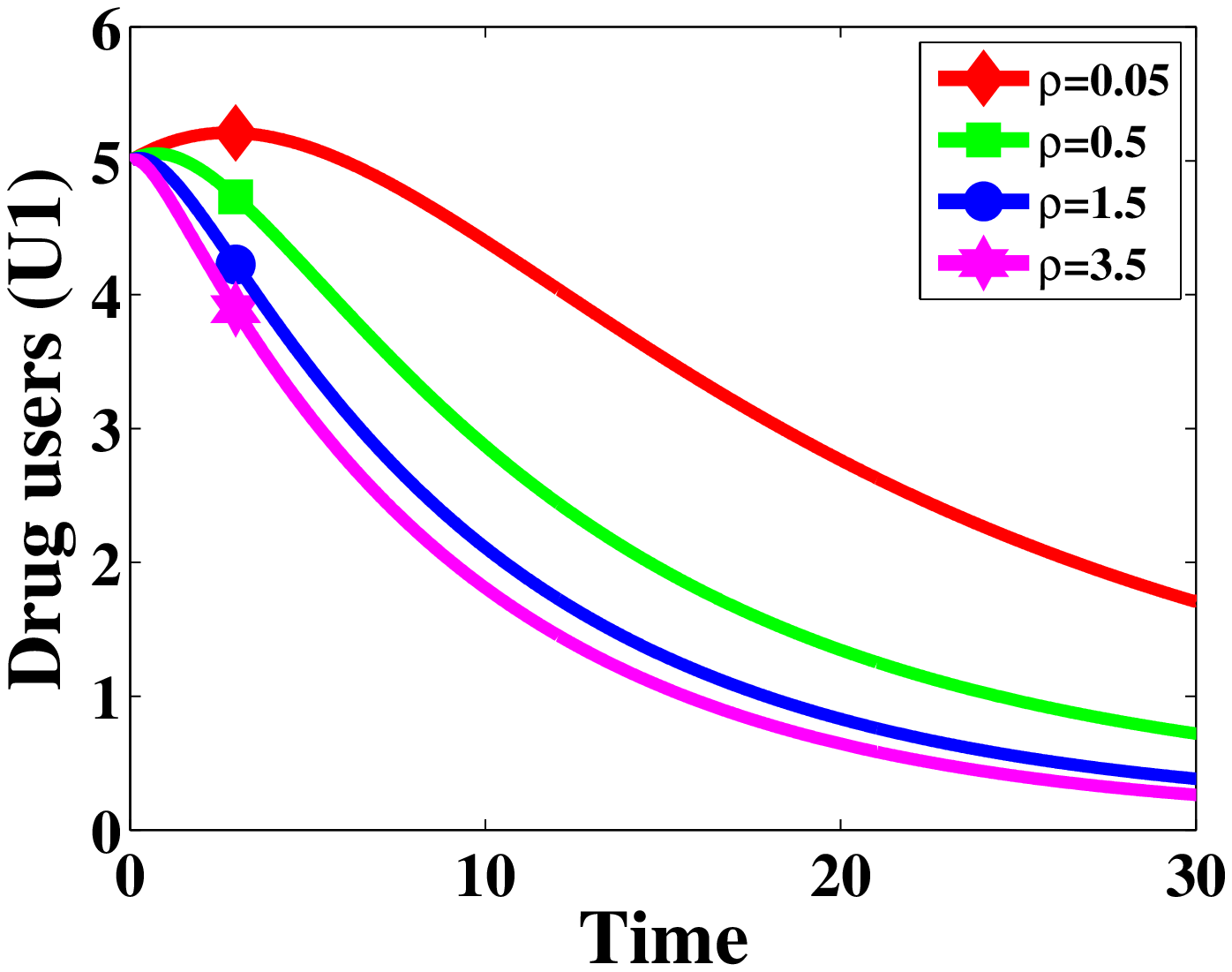}
\caption{Drug users}
\label{fig:5a}
\end{subfigure}
\begin{subfigure}[t]{0.48\textwidth}
\centering
\includegraphics[scale=0.40]{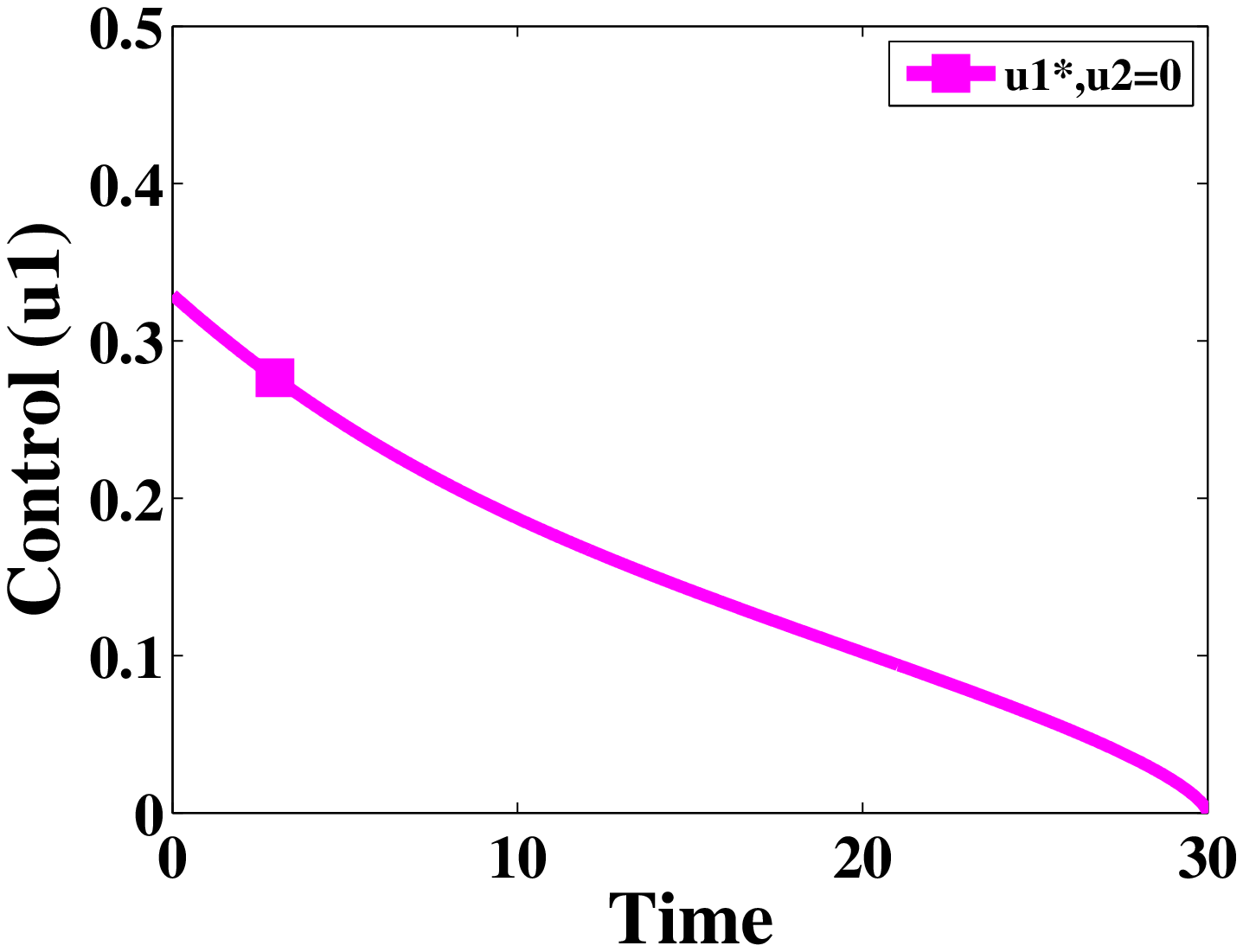}
\caption{Optimal intensity $u_1^*$}
\label{fig:5b}
\end{subfigure}
\caption{Drug users and optimal intensity of $u_1$ in Case~1
for various values of the rate $\rho$ of information interaction.} 
\label{fig:d-different}
\end{figure}
It clearly shows that an increase in the information interaction 
rate $\rho$ decreases the drug user population. Moreover, 
the control profile is plotted in Fig.~\ref{fig:5b}.

{\bf{Case 2:}} We are continuing the numerical simulations 
using the same parameter values as in Case~1 but with control $u_2$. 
Then the evolution of the corresponding population densities 
of the optimal control system is depicted in  Fig.~\ref{fig:only:u1}. 
Here we can understand that the number of drug users in treatment 
is rapidly increasing over the course of time when compared 
with Case~1, see Fig.~\ref{fig:c}. 
Further, the influence of control $u_2$ 
is also there in other population densities of the model, 
see Figs.~\ref{fig:a} and \ref{fig:b}. 
The corresponding optimal control profile is given 
in Fig.~\ref{fig:6a}. It shows that treatment for drug users 
rapidly decreases in the stipulated time and then goes to zero. 
We conclude that medical treatment plays a crucial role 
to reduce the population of heroin users. 
\begin{figure}[h!]
\centering
\begin{subfigure}[t]{0.48\textwidth}
\centering
\includegraphics[scale=0.40]{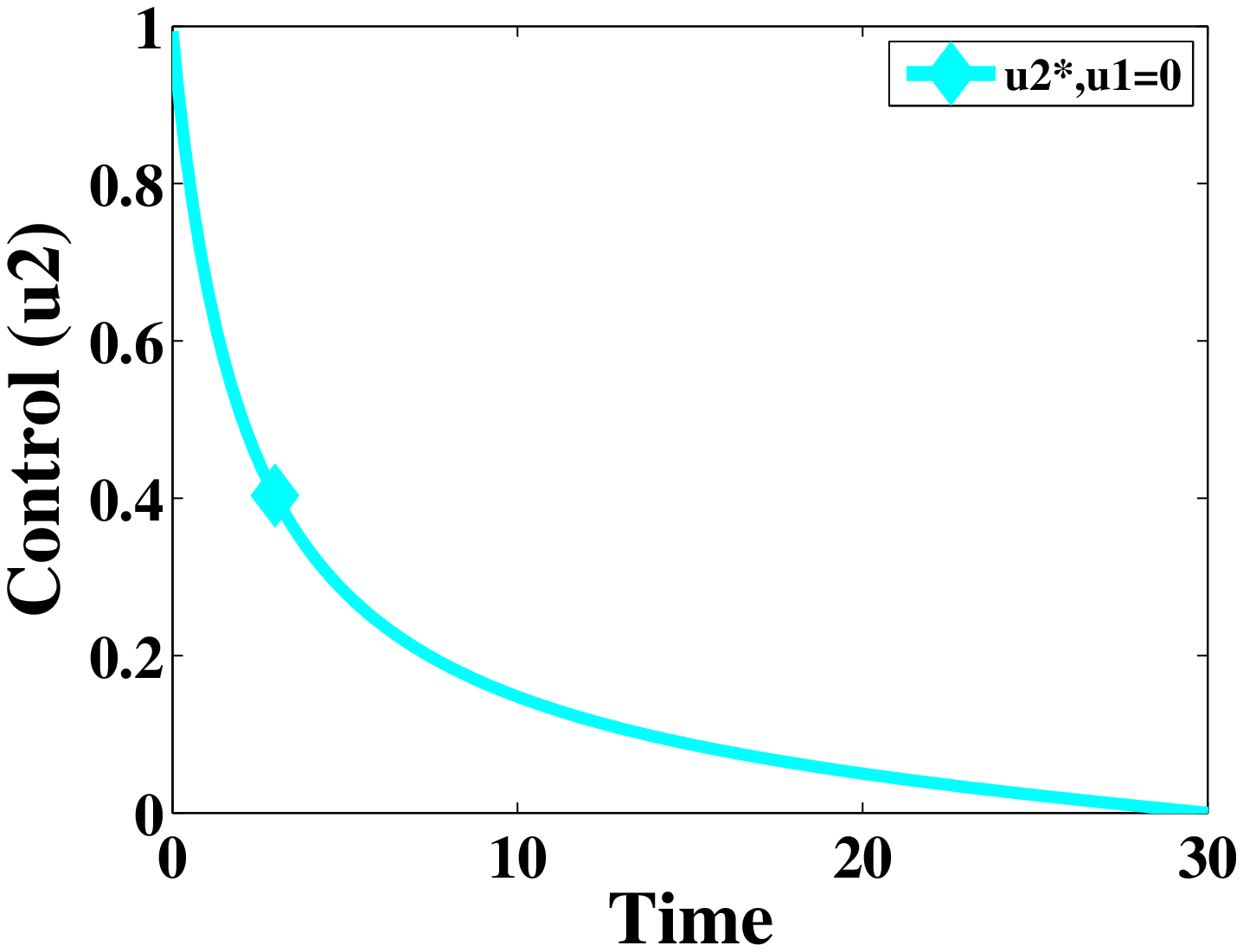}
\caption{Optimal intensity of $u_2^*$ (Case~2)}
\label{fig:6a}
\end{subfigure}
\begin{subfigure}[t]{0.48\textwidth}
\centering
\includegraphics[scale=0.40]{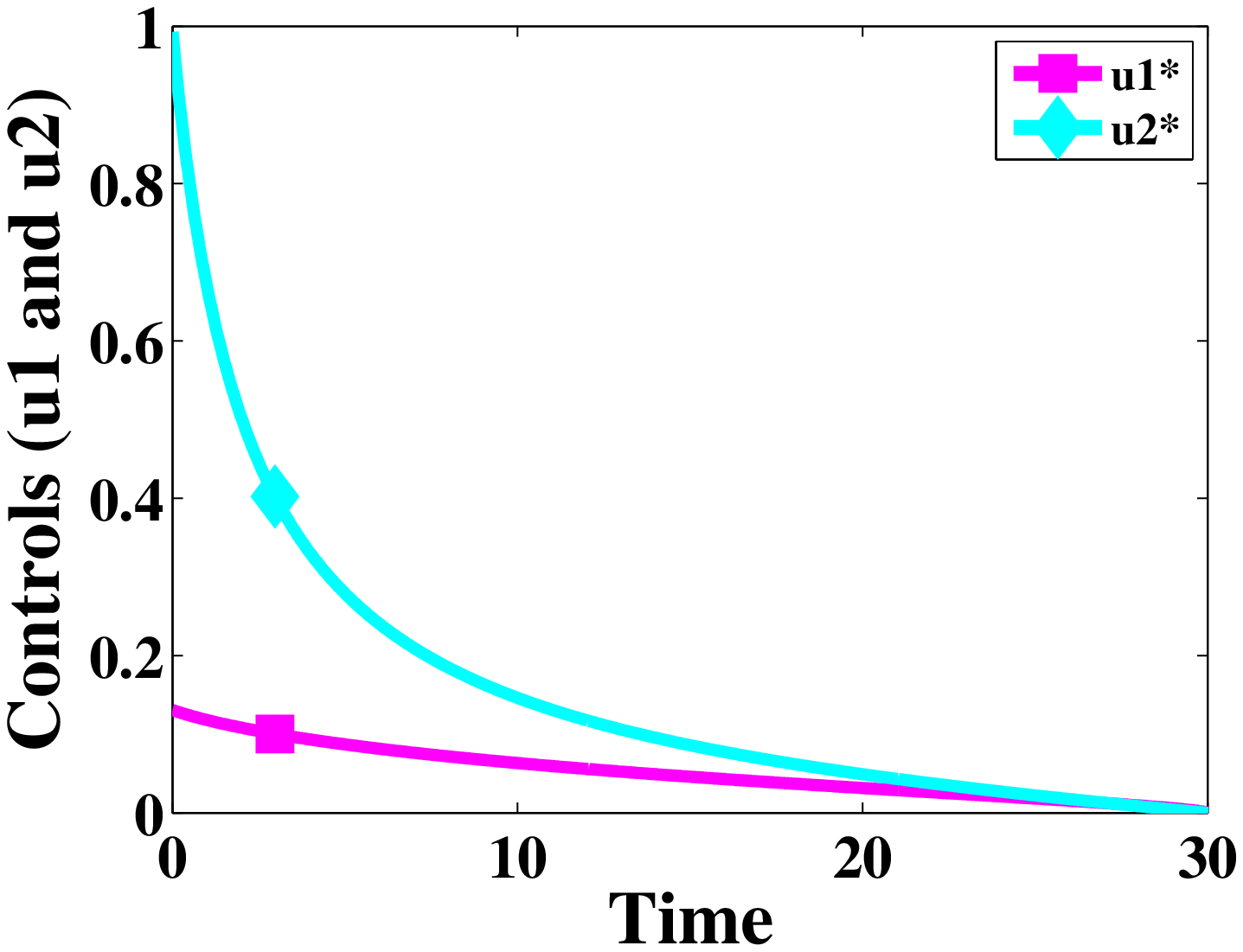}
\caption{Optimal controls $u_1^*$ and $u_2^*$ (Case~3)}
\label{fig:6b}
\end{subfigure}
\caption{Optimal control profiles.} 
\label{fig:bothcontrol}
\end{figure}
        
{\bf{Case 3:}}
In this case, we take non-zero control interventions, that is, 
we compute the solution of the optimal control problem exactly
as discussed in Section~\ref{s4}. Further, we continue the simulations 
with the same parameter values as in the previous two cases. 
The evolution of the population densities is depicted in 
Fig.~\ref{fig:only:u1}. As expected, the influence of both controls 
is more effective than the other two cases already discussed. 
The optimal control profile is depicted in Fig.~\ref{fig:6b}. 
Fig.~\ref{fig:Jbothcontrol} shows the information level 
with optimal policies versus without any control measures. 
\begin{figure}[h!]
\centering
\includegraphics[scale=0.50]{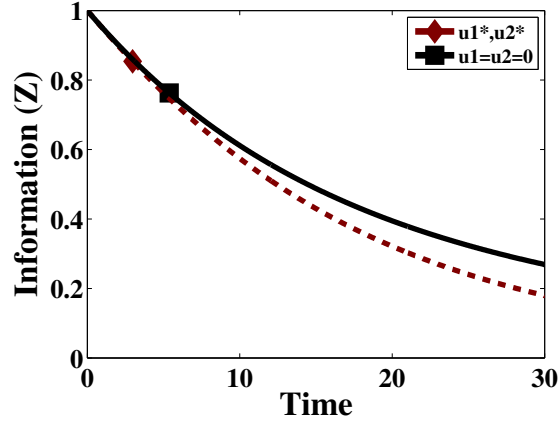}
\caption{Information level with $u_1^*,u_2^*$ \eqref{oc} versus $u_1=u_2=0$.} 
\label{fig:Jbothcontrol}
\end{figure}
\begin{figure}[h!]
\centering	
\begin{subfigure}[t]{0.48\textwidth}
\centering
\includegraphics[scale=0.40]{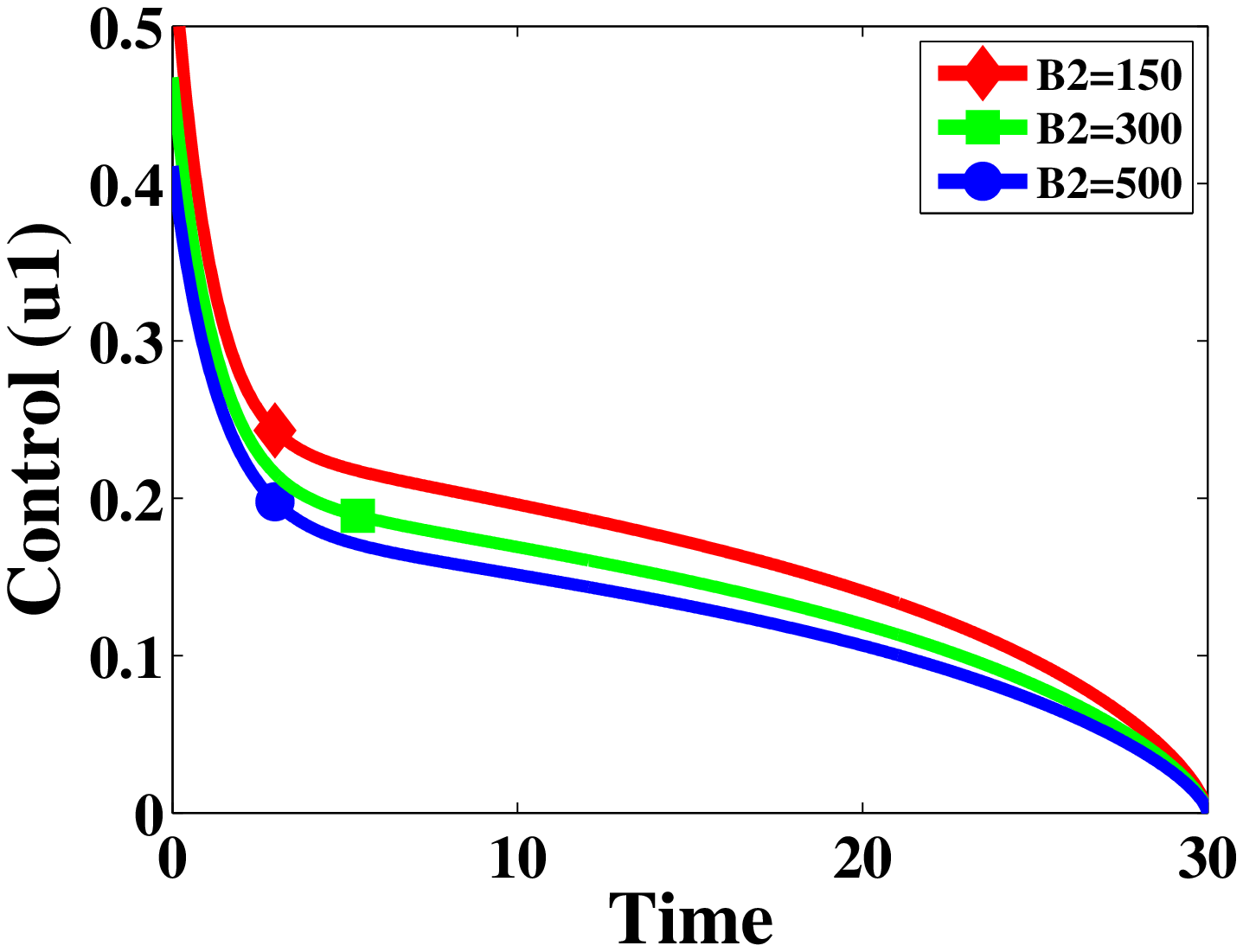}
\caption{Changing $B_2$ ($B_1 = 6$, $B_3 = 30$)}
\label{fig:8a}
\end{subfigure}
\begin{subfigure}[t]{0.48\textwidth}
\centering
\includegraphics[scale=0.40]{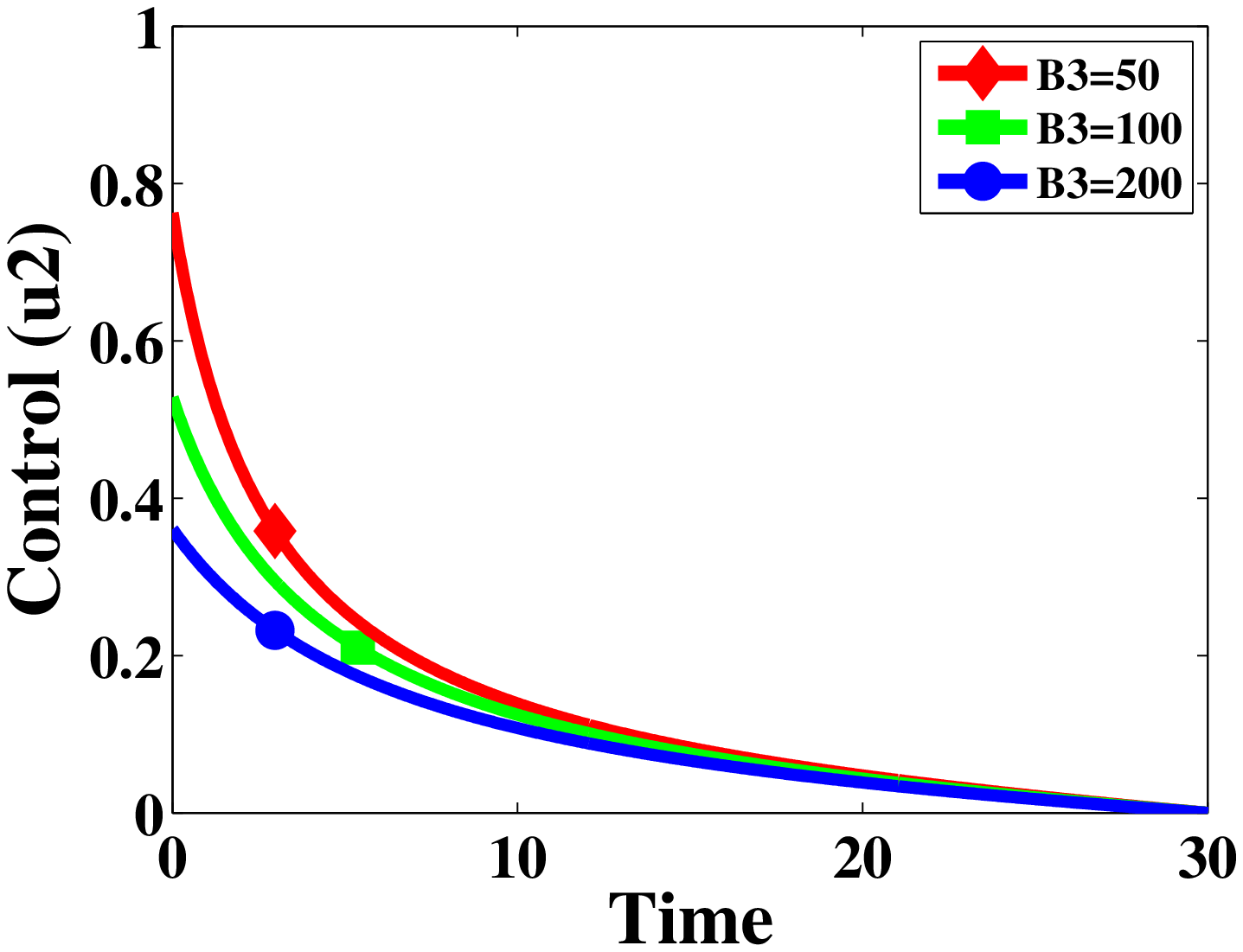}
\caption{Changing $B_3$ ($B_1 = 6$, $B_2 = 120$)}
\label{fig:8b}
\end{subfigure}
\caption{Optimal controls $u_1^*,u_2^*$ \eqref{oc} 
with different weights in cost functional \eqref{cost}.} 
\label{fig:controlsdifferent}
\end{figure}
Furthermore, we also perform numerical simulations 
with different weights for both control interventions $u_1$ and $u_2$. 
The result given in Fig.~\ref{fig:8a} shows how the control strategies 
depend on weight $B_2$. It is noted that if the positive weight $B_2$ increases, 
then the amount of control policy $u_1$ decreases. Figure~\ref{fig:8b} 
illustrates how the control strategies depend on weight $B_3$. The amount 
of treatment $u_2$ decreases as the positive weight $B_3$ increases.


\section{Conclusion}
\label{sec:conc}

We examined an optimal control problem for a heroin epidemic model. 
Information regarding prevention education and drug treatments were 
considered as control interventions. Both controls have their advantage 
and efficiency in implementation. Stability theory was used to analyze 
the mathematical model qualitatively. The system has two equilibrium points: 
a drug-free equilibrium, which always exists, and an endemic equilibrium, 
which exists when the basic reproduction number is greater than one. 
We analytically found controls 
in terms of state and costate variables and then numerically 
solved the boundary value problem for the 
resulting system of ordinary differential equations, 
finding the optimal paths. Further, various control strategies 
were studied numerically for the proposed control problem. Finally, 
we concluded that prevention programs and treatment not only decrease 
the cost burden but also minimize the number of drug abuse cases.  
As a future direction of research, one can investigate the proposed
heroin model by introducing stochastic effects on the unknowns \cite{MR4173153}.
Further, application of several types of delays \cite{MR4236351}
and multiobjective optimization \cite{MR3804169} are also pointed 
out as interesting directions for future research. 


\section*{Acknowledgments}

The authors are grateful to two anonymous reviewers
for several constructive comments that really helped 
to improve the manuscript. 


\section*{Funding}

Sowndarrajan is thankful to the Ministry of Human Resources Development (MHRD) 
and National Institute of Technology Goa, India, for awarding him a Senior Research Fellowship. 
Debbouche and Torres are grateful to the Portuguese Foundation for Science and Technology (FCT),  
project UIDB/04106/2020 (CIDMA).




\begin{thebibliography}{xx}

\bibitem{NIDA} 
NIDA InfoFacts: Heroin. 
\url{http://www.nida.nih.gov/infofacts/heroin.html}

\bibitem{WDR2018} 
UNODC, 
World Drug Report (2018).
\url{https://www.unodc.org/wdr2018}

\bibitem{Lipari2015}  
Lipari, R.N.,  Hughes, A.: 
Trends in heroin use in the United States: 2002 to 2013, (2015)

\bibitem{Li}  
Li, X.,  Zhou, Y.,  Stanton, B.: 
Illicit drug initiation among institutionalized drug users in China, 
Addiction,  {\bf{97}}, 575--582 (2002)

\bibitem{Garten}  
Garten, R.J., Lai,  S.,  Zhang, J.,  Liu, W.,  Chen, J.,  Vlahov, D.,  Yu, X.F.: 
Rapid transmission of hepatitis C virus among 
young injecting heroin users in Southern China, 
Int. J. Epidemiol. {\bf{33}}, 182--188 (2004)

\bibitem{Mulone}  
Mulone, G.,  Straughan, B.: 
A note on heroin epidemics, 
Math. Biosci. {\bf{218}}, 138--141 (2009)

\bibitem{White}  
White, E., Comiskey, C.: 
Heroin epidemics, treatment and ODE modelling, 
Math. Biosci. {\bf{208}}, 312--324 (2007)

\bibitem{Samanta2011}  
Samanta, G.P.: 
Dynamic behaviour for a nonautonomous heroin epidemic model with time delay, 
J. Appl. Math. Comput. {\bf{35}}, 161--178 (2011)

\bibitem{Mushayabasa2015}  
Mushayabasa, S., Tapedzesa,  G.: 
Modeling illicit drug use dynamics and its optimal control analysis, 
Comput. Math. Methods Med. {\bf{2015}}, Art. ID 383154, 11~pp (2015)

\bibitem{Mushayabasa2011}  
Mushayabasa, S.,  Bhunu, C.P.: 
Epidemiological consequences of non-compliance 
to HCV therapy among intravenous drug users, 
International Journal of Research and Reviews 
in Applied Sciences {\bf{8}}, 288--295 (2011)

\bibitem{UNODC2014} 
United Nations Office on Drugs and Crime (UNODC), 
World Drug Report (2014)

\bibitem{Isaac2017}  
Wangari, I.M., Stone, L.: 
Analysis of a heroin epidemic model 
with saturated treatment function, 
J. Appl. Math., {\bf{2017}}, Art. ID 1953036, 21~pp (2017)
 
\bibitem{Wang2011} 
Wang, X.,  Yang, J.,  Li, X.: 
Dynamics of a heroin epidemic model with very population, 
Applied Mathematics {\bf{2}}, 732--738  (2011)

\bibitem{Huang2013} 
Huang, G.,  Liu, A.: 
A note on global stability for a heroin epidemic model with distributed delay, 
Appl. Math. Lett. {\bf{26}}, 687--691 (2013)

\bibitem{Liu2011} 
Liu, J.,  Zhang, T.: 
Global  behaviour of a heroin epidemic model with distributed delays, 
Appl. Math. Lett. {\bf{24}}, 1685--1692 (2011)

\bibitem{MR3877418}
Wang, J., Wang, J., Kuniya, T.: 
Analysis of an age-structured multi-group heroin epidemic model, 
Appl. Math. Comput. {\bf{347}}, 78--100 (2019) 

\bibitem{Saha}  
Saha, S.,  Samanta, G.P.: 
Synthetic drugs transmission: stability analysis and optimal control, 
Lett. Biomath.  1--31 (2019)

\bibitem{Kassa2015}  
Kassa, S.,  Ouhinou, A.: 
The impact of self-protective measures in the optimal interventions 
for controlling infectious diseases of human population, 
J. Math. Biol. {\bf{70}}, 213--236 (2015) 

\bibitem{Saha2019}  
Saha, S.,  Samanta, G.P.: 
Modelling and optimal control of HIV/AIDS prevention 
through PrEP and limited treatment, 
Physica A. {\bf{516}}, 280--307 (2019)

\bibitem{Joshi2015}  
Joshi, H.,  Lenhart, S.,  Hota, S.,  Agusto, F.: 
Optimal control of an SIR model with changing 
behavior through an education campaign, 
Electron. J. Differential Equations. {\bf{50}}, 1--14 (2015)

\bibitem{Nicholas2017}  
Battista, N.A.,  Pearcy, L.B.,  Strickland, W.C.: 
Modeling the opioid epidemic,
Bull. Math. Biol. {\bf{81}}, 2258--2289 (2019)

\bibitem{Abdulfatai2018}  
Momoh, A.A.,  F${u}$genschuh, A.: 
Optimal control of intervention strategies 
and cost effectiveness analysis for a Zika virus model, 
Operations Research for Health Care. {\bf{18}}, 99--111 (2018)

\bibitem{Ebenezer2016} 
Bonyah,  E.,  Badu, K.,  Asiedu-Addo, S.K.: 
Optimal control application to an Ebola model, 
Asian Pac. J. Trop. Biomed. {\bf{6}}, 283--289 (2016)

\bibitem{MyID:364}
Area, I., Nda\"{\i}rou, F., Nieto, J.J., Silva, C.J., Torres, D.F.M.:
Ebola model and optimal control with vaccination constraints,
J. Ind. Manag. Optim. {\bf 14}, 427--446 (2018)
{\tt arXiv:1703.01368}

\bibitem{MR4178232}
Khan, A., Zaman, G., Ullah, R., Naveed, N.: 
Optimal control strategies for a heroin epidemic model 
with age-dependent susceptibility and recovery-age, 
AIMS Math. {\bf{6}}, no.~2, 1377--1394 (2021)

\bibitem{MR4261505}
Khan, A., Zaman, G., Ullah, R., Naveed, N.: 
Correction: Optimal control strategies for a heroin epidemic model 
with age-dependent susceptibility and recovery-age, 
AIMS Math. {\bf{6}}, no.~7, 7318--7319 (2021)

\bibitem{Lakshmikantham} 
Lakshmikantham,  V.,  Leela, S., Martynyuk,  A.A.: 
Stability Analysis of Nonlinear Systems, 
Marcel Dekker, Inc., New York, Basel (1989)

\bibitem{van2002}  
van den Driessche, P.,  Watmough, J.: 
Reproduction numbers and sub-threshold endemic 
equilibria for compartmental models of disease transmission, 
Math. Biosci. {\bf{180}}, 29--48 (2002)

\bibitem{Arriola2005}  
Arriola, L.,  Hyman, J.: 
Lecture notes, forward and adjoint sensitivity analysis: 
with applications in Dynamical Systems, 
Linear Algebra Optim. Math. Theor. Biol. Inst., {\bf{2005}}

\bibitem{Rosa2019}
Rosa, S., Torres, D.F.M.: 
Optimal control and sensitivity analysis of a fractional order TB model, 
Stat. Optim. Inf.Comput. {\bf 7}, 617--625 (2019).
{\tt arXiv:1812.04507}

\bibitem{Silva2013}  
Silva, C.J., Torres, D.F.M.: 
Optimal control for a tuberculosis model 
with reinfection and post-exposure interventions, 
Math. Biosci. {\bf{244}}, 154--164 (2013)
{\tt arXiv:1305.2145}

\bibitem{Grass} 
Zeiler,  I.,  Caulkins, J., Grass,  D.,  Tragler, G.: 
Keeping options open: an optimal control model 
with trajectories that reach a dnss point in positive time, 
SIAM J. Control Optim. {\bf{48}}, 3698--3707 (2010)

\bibitem{kumar2017}  
Kumar, A.,  Srivastava, P.K.:  
Vaccination and treatment as control interventions 
in an infectious disease model with their cost optimization, 
Commun.  Nonlinear Sci. Numer. Simul. {\bf{44}}, 334--343 (2017)

\bibitem{Coddington} 
Coddington,  E.,  Levinson, N.: 
Theory of Ordinary Differential Equations, 
Tata McGraw-Hill Education (1955)

\bibitem{Gaff2011}  
Gaff, H.,  Schaefer, E.,  Lenhart, S.: 
Use of optimal control models to predict treatment time 
for managing tick-borne disease, 
J. Biol. Dyn. {\bf{5}}, 517--530 (2011)

\bibitem{Lenhart}  
Lenhart, S.M., Workman, J.T.: 
Optimal Control Applied to Biological Models, 
CRC Press, Boca Raton, FL (2007)

\bibitem{Pontryagin}  
Pontryagin, L.S., Boltyanskii, V.G., Gamkrelidze, R.V., Mishchenko, E.F.: 
The mathematical theory of optimal processes,
A Pergamon Press Book, The Macmillan Co., New York (1964)

\bibitem{MR4173153}
Zine, H., Boukhouima, A., Lotfi, E.M., Mahrouf, M., Torres, D.F.M., Yousfi, N.: 
A stochastic time-delayed model for the effectiveness of Moroccan 
COVID-19 deconfinement strategy, 
Math. Model. Nat. Phenom. {\bf 15}, Paper 50, 14~pp (2020)
{\tt arXiv:2010.16265}

\bibitem{MR4236351}
Abraha, T., Al Basir, F., Obsu, L.L., Torres, D.F.M.: 
Pest control using farming awareness: Impact of time delays 
and optimal use of biopesticides, 
Chaos Solitons Fractals {\bf 146}, Art.~110869, 11~pp (2021) 
{\tt arXiv:2103.06007}

\bibitem{MR3804169}
Denysiuk, R., Silva, C.J., Torres, D.F.M.: 
Multiobjective optimization to a TB-HIV/AIDS 
coinfection optimal control problem, 
Comput. Appl. Math. {\bf 37}, 2112--2128 (2018)
{\tt arXiv:1703.05458}

\end{thebibliography}
\end{document}